\documentclass{amsart}
\usepackage{amssymb}
\usepackage{amsmath}
\usepackage{mathtools}
\usepackage{pdfpages}
\usepackage{changes}
\usepackage{tikz-cd}

\newcommand{\N}{\mathbb{N}}
\newcommand{\Z}{\mathbb{Z}}

\newcommand{\D}{\mathbb{D}}
\newcommand{\Sym}{\mathbb{S}}

\newcommand{\Aut}{\operatorname{Aut}}

\newcommand{\GL}{\mathrm{GL}}
\newcommand{\id}{\mathrm{id}}
\newcommand{\Soc}{\mathrm{Soc}}
\newcommand{\Fix}{\mathrm{Fix}}

\makeatletter
\numberwithin{equation}{section}
\numberwithin{figure}{section}
\numberwithin{table}{section}
\newtheorem{thm}{Theorem}[section]
\newtheorem{lem}[thm]{Lemma}
\newtheorem{cor}[thm]{Corollary}
\newtheorem{pro}[thm]{Proposition}
\newtheorem{defn}[thm]{Definition}

\newtheorem{example}[thm]{Example}
\newtheorem{question}[thm]{Question}
\newtheorem{convention}[thm]{Convention}
\newtheorem{rem}[thm]{Remark}
\newtheorem{exa}[thm]{Example}

\makeatother

\begin{document}


\title{Skew left braces of nilpotent type}

\author{Ferran Ced\'o}
\author{Agata Smoktunowicz}
\author{Leandro Vendramin}
\address{Departament de Matem\`atiques, Universitat Aut\`onoma de Barcelona, 08193 Bellaterra (Barcelona), Spain}
\email{cedo@mat.uab.cat}

\address{School of Mathematics, The University of Edinburgh, James Clerk Maxwell Building, The Kings Buildings, Mayfield Road EH9 3JZ, Edinburgh}
\email{A.Smoktunowicz@ed.ac.uk}

\address{IMAS--CONICET and Departamento de Matem\'atica, FCEN, Universidad de Buenos Aires, Pabell\'on~1, Ciudad Universitaria, C1428EGA, Buenos Aires, Argentina}
\email{lvendramin@dm.uba.ar}




\maketitle

\begin{abstract}
    We study series of left ideals of skew left braces that are
    analogs of upper central series of groups. These concepts allow us to
    define left and right nilpotent skew left braces. Several results related to these concepts are proved
    and applications to infinite left braces are given.
    Indecomposable solutions of the Yang--Baxter equation are explored using
    the structure of skew left braces.
\end{abstract}

\setcounter{tocdepth}{1}
\tableofcontents


\section*{Introduction}

The Yang--Baxter equation was first introduced in the field of
statistical mechanics. It depends on the idea that in some
scattering situations particles may preserve their momentum while
changing their quantum internal states. The equation states that a
matrix $R$ satisfies
\[
(R\otimes\id)(\id\otimes R)(R\otimes\id)
=(\id\otimes R)(R\otimes\id)(\id\otimes R).
\]
In one dimensional quantum systems, $R$ is the scattering matrix and
if it satisfies the Yang--Baxter equation then the system is
integrable. The Yang--Baxter equation also appears in topology and
algebra mainly because its connections with braid groups.  It takes
its name from independent work of Yang~\cite{MR0261870} and
Baxter~\cite{MR0290733}.

In~\cite{MR1183474} Drinfeld observed that the Yang--Baxter equation
also makes sense even if $R$ is not a linear operator $V\otimes V\to
V\otimes V$ but a map $X\times X\to X\times X$ where $X$ is a set.
In this context, non-degenerate solutions are interesting. Recall
that a set-theoretic solution of the Yang--Baxter equation $(X,r)$
is said to be \emph{non-degenerate} if
$r(x,y)=(\sigma_x(y),\tau_y(x))$ for all $x,y\in X$, where
$\sigma_x$ and $\tau_y$ are permutations of $X$.  Permutation
solutions are examples of non-degenerate solutions: if $X$ is a
set, $\sigma$ and $\tau$ are permutations of $X$ such that
$\sigma\tau=\tau\sigma$ and $r(x,y)=(\sigma(y),\tau(x))$, then the
pair $(X,r)$ is a non-degenerate set-theoretic solution of the
Yang--Baxter equation.

The first papers where non-degenerate set-theoretic solutions were studied
are~\cite{MR1722951,MR1637256,MR1769723,MR1809284}.  It turns out that
set-theoretic solutions have connections for example with groups of I-type,
Bieberbach groups, bijective $1$-cocycles, Garside theory, etc.
In~\cite{MR2278047} Rump found a new and unexpected connection between
set-theoretic solutions and Jacobson radical rings: such rings produce
involutive solutions. This observation was the key for introducing a new
algebraic structure that generalizes Jacobson radical rings.  To strengthen the
connection with rings, Rump conceived the name \emph{braces}. New connections
appeared, for example with regular subgroups and Hopf--Galois
extensions~\cite{MR3465351}, left orderable groups~\cite{MR3815290}, flat
manifolds~\cite{MR3291816}.

In~\cite{MR3647970} braces were generalized to skew left braces and
this structure was used to produce and study not necessarily
involutive solutions. Skew braces are useful for studying regular
subgroups and Hopf--Galois extensions, bijective $1$-cocycles, rings
and near-rings, triply factorized groups, see for
example~\cite{MR3763907}.

A skew left brace is a set with two compatible group structures. One of these
groups is known as the \emph{multiplicative group}; the other as the
\emph{additive group}. The terminology used in the theory of Hopf--Galois
extensions suggest that the additive group determines the \emph{type} of the
skew left brace. For example, skew left braces of abelian type are Rump's
braces, i.e. those braces with abelian additive group.

A skew left brace is a triple $(A,+,\circ)$, where $(A,+)$ and $(A,\circ)$ are
(not necessarily abelian)  groups such that the compatibility
\[
a\circ (b+c)=a\circ b-a+a\circ c
\]
holds for all $a,b,c\in A$.

Radical rings are certain examples of skew left braces where the additive group
$(A,+)$ is abelian.  For $a,b\in A$ one defines
\[
    a*b=-a+a\circ b-b.
\]
In the case of Jacobson radical rings, the operation $*$ is the multiplication
of the ring. For arbitrary skew left braces the operation $*$ is not
associative, but nevertheless the relation with radical rings suggests how to
translate ideas from ring theory to the world of skew left braces.

In~\cite[Definition 3.1]{MR1722951} Etingof, Schedler and Soloviev introduced
multipermutation solutions. As the name suggests, such solutions are
generalizations of permutation solutions.  Gateva--Ivanova's strong
conjecture~\cite[2.28(I)]{MR2095675} states that square-free involutive
non-degenerate set-theoretic solutions are multipermutation solutions.  Despite
the conjecture was proved to be false~\cite{MR3437282}, it was the motivation
of several original
investigations~\cite{MR2652212,MR3177933,MR2885602,MR3439888}.

Braces provide a powerful algebraic framework to work with set-theoretic
solutions.  In \cite[Theorem 3.1]{MR3647970} it is proved that for a skew left
brace $A$, the map
\[
    r_A\colon A\times A\to A\times A,\quad
    r_A(a,b)=(-a+a\circ b,(-a+a\circ b)'\circ a\circ b),
\]
is a non-degenerate
set-theoretic solution of the Yang--Baxter equation.
In~\cite[Theorem 4.5]{MR3763907} it is proved that
if $(X,r)$ be a non-degenerate solution of the Yang--Baxter equation,
then there exists a unique skew left brace structure over the group
\[
    G=G(X,r)=\langle X:x\circ y=u\circ v\text{ whenever $r(x,y)=(u,v)$}\}
\]
such that
  \[
  \begin{tikzcd}
         X\times X\arrow[r, "r"] \arrow["\iota\times\iota"', d]
         & X\times X \arrow[d, "\iota\times\iota" ] \\
         G\times G \arrow[r, "r_{G}"]
         & G\times G
  \end{tikzcd}
  \]
where $\iota\colon X\to G(X,r)$ is the canonical map.
Moreover, the pair $(G(X,r),\iota)$ has the following universal property: if $B$
is a skew left brace and $f\colon X\to B$ is a map such that
\[
  \begin{tikzcd}
         X\times X\arrow[r, "r"] \arrow["f\times f"', d]
         & X\times X \arrow[d, "f\times f" ] \\
         B\times B \arrow[r, "r_{B}"]
         & B\times B
  \end{tikzcd}
\]
then there exists a unique skew left brace homomorphism
$\phi\colon G(X,r)\to B$ such that
\[
\begin{tikzcd}
       X \arrow[rd, "f"'] \arrow[r, "\iota"] & G\arrow[d, "\phi"] \\
       & B
\end{tikzcd}
\qquad\text{and}\qquad
\begin{tikzcd}
       G\times G\arrow[r, "r_{G}"] \arrow["\phi\times\phi"', d]
         & G\times G \arrow[d, "\phi\times\phi" ] \\
         B\times B \arrow[r, "r_{B}"]
         & B\times B
  \end{tikzcd}
  \]
commute.

To study Gateva--Ivanova's strong conjecture Rump introduced two sequences of
left ideals~\cite{MR2278047}.  These sequences turned out to be very important
for understanding the structure of Rump's braces.  It makes sense to consider
similar sequences in the context of skew left braces.  The right series of a
skew left brace $A$ is defined as the sequence
\[
A\supseteq A^{(2)}\supseteq A^{(3)}\supseteq\cdots,
\]
where $A^{(n+1)}=A^{(n)}*A$ denotes the additive subgroup of $A$ generated by
elements of the form $x*a$ for $a\in A$ and $x\in A^{(n)}$. Each $A^{(n)}$ is
an ideal of $A$, see Proposition~\ref{pro:right_series}.  A skew left brace
is said to be \emph{right nilpotent} if there is a positive integer $n$ such
that $A^{(n)}=0$.
The left series of a
skew left brace $A$ is the sequence
\[
A\supseteq A^2\supseteq A^3\supseteq\cdots,
\]
where $A^{n+1}=A*A^n$.  Each $A^n$ is a left ideal of $A$, see
Proposition~\ref{pro:left_series}.
A skew left brace
is said to be \emph{left nilpotent} if there is a positive integer $n$ such
that $A^{n}=0$.
Following~\cite{MR3814340} we also define the sequence
\[
    A\supseteq A^{[2]}\supseteq A^{[3]}\supseteq\cdots,
\]
where $A^{[1]}=A$ and $A^{[n+1]}$ is the additive subgroup of $A$ generated by
all the elements from the sets $A^{[i]}*A^{[n+1-i]}$ for $1\leq i\leq n$.  Each
$A^{[n]}$ is a left ideal of $A$, see Proposition~\ref{pro:Smoktunowicz}. A
skew left brace $A$ is said to be \emph{strongly nilpotent} if there is a
positive integer $n$ such that $A^{[n]}=0$. Theorem~\ref{thm:equivalence}
states that a skew left brace $A$ is strongly nilpotent if and only if it is
left and right nilpotent.

The sequence of the $A^{[n]}$ is a basic tool for understanding problems
similar to K\"othe conjecture in the context of skew left braces, see
Section~\ref{nilpotency}.

One of our main goals is to study the connection between left and
right nilpotency and the structure of skew left braces.  We study
the connection between right nilpotent skew left braces and
multipermutation solutions (see
Theorem~\ref{thm:mpl&right_nilpotent}).  We also explore the cases
where the groups of the skew left brace are nilpotent.  If the skew
left brace is finite and both groups are nilpotent, then it is
possible to write the skew left brace as a direct product of skew
left braces of prime-power order (see Corollary~\ref{cor:product}).
This result is then applied to study infinite skew left braces
with multiplicative group isomorphic to $\Z$ (see
Theorems~\ref{thm:Z} and~\ref{thm:ZrightNilpotent}); this
answers~\cite[Question A.10]{MR3763907}.  Left nilpotent skew left
braces are also studied. We show that a finite skew left brace with nilpotent additive 
group is left nilpotent if and only if its multiplicative
group is nilpotent (see Theorem~\ref{thm:left_nilpotent=nilpotent}).

\medskip
This paper is organized as follows. In Section~\ref{preliminaries}
basic definitions are recalled. These definitions include skew left
braces, left ideals and ideals.  In Section~\ref{nilpotency} we
define left and right nilpotent skew left braces. In
Theorem~\ref{thm:IcapSoc} and Proposition~\ref{thm:Fix} we prove
analogs of Hirsch's theorem for nilpotent groups. We use these
results in Theorem~\ref{thm:mpl&right_nilpotent} to prove that,
under mild assumptions, a skew left brace is right nilpotent if and
only if it has finite multipermutation level (see
Theorem~\ref{thm:mpl&right_nilpotent}). In Section~\ref{perfect} we
deal with perfect skew left braces. Using wreath product of skew
left braces to show that one cannot produce an analog of Gr\"un's
lemma for skew left braces.  Skew left braces of nilpotent type are
studied in Section~\ref{nilpotent_type}. Our first result is
Corollary~\ref{cor:product}, where we prove that if both groups of a
finite skew left  brace are nilpotent, then the skew left brace is a
direct product of sub skew left braces of prime-power size.
Theorem~\ref{thm:A2} is the consequence of applying Hall's results
from \cite{Hall} to the case of skew left braces. We prove in
Theorem~\ref{thm:left_nilpotent=nilpotent} that a finite skew left
brace of nilpotent type is left nilpotent if and only if its
multiplicative group is nilpotent. In particular, skew left braces
of prime-power size are left nilpotent. In Section~\ref{cyclic} we
prove that skew left braces of abelian type with infinite-cyclic
multiplicative group are trivial, see Theorem~\ref{thm:Z}.
Theorem~\ref{thm:infinite_dihedral} shows that this result cannot be
extended to skew left braces. Finally, in
Section~\ref{indecomposable} some of our results are used for
studying indecomposable set-theoretic solutions. In this section we
give a positive answer to~\cite[Question~5.6]{MR3771874}.

\section{Preliminaries}
\label{preliminaries}

Recall that a \emph{skew left brace} is a triple $(A,+,\circ)$, where $(A,+)$ and
$(A,\circ)$ are (not necessarily abelian) groups such that
\[
a\circ (b+c)=a\circ b-a+a\circ c
\]
holds for all $a,b,c\in A$. The group $(A,\circ)$ will be the
\emph{multiplicative group} of the skew left brace and $(A,+)$ will be the
\emph{additive group} of the skew left brace.  We write $a'$ to denote the inverse
of $a$ with respect to the circle operation $\circ$.
A skew left brace $(A,+,\circ)$ such that $a\circ b=a+b$ for all $a,b\in A$ is
said to be \emph{trivial}.

\begin{defn}
    Let $\mathcal{X}$ be a property of groups. A skew left brace $A$ is said to be
    of $\mathcal{X}$-type if its additive group belongs to $\mathcal{X}$.
\end{defn}

Rump's braces are skew left braces of abelian type.

\begin{convention}
    Skew left braces of abelian type will be called \emph{left braces}.
\end{convention}

If $A$ is a skew left brace, the multiplicative group
acts on the additive group by automorphisms. The map $\lambda\colon
(A,\circ)\to \Aut(A,+)$, $a\mapsto\lambda_a$, where $\lambda_a(b)=-a+a\circ b$,
is a group homomorphism, see~\cite[Corollary 1.10]{MR3647970}.

\begin{rem}
Let $A$ be a skew left brace. Then the following formulas hold:
\begin{align*}
    &a\circ b = a+\lambda_a(b),
    &&a+b=a\circ \lambda^{-1}_a(b),
    &&\lambda_a(a')=-a.
\end{align*}
\end{rem}

Let $A$ be a skew left brace. For $a,b\in A$ let
\[
    a*b=\lambda_a(b)-b=-a+a\circ b-b.
\]
The following identities are easily verified:
\begin{align}
&a*(b+c)=a*b+b+a*c-b,\\
&(a\circ b)*c=(a*(b*c))+b*c+a*c,
\end{align}
These identities are similar to the usual
commutator identities.

\begin{thm}
Let $A$ be an additive (not necessarily abelian) group.  A skew left
brace structure over $A$ is equivalent to an operation $A\times A\to
A$, $(a,b)\mapsto a*b$, such that $a*(b+c)=a*b+b+a*c-b$ holds for
all $a,b,c\in A$, and the operation $a\circ b=a+a*b+b$ turns $A$
into a group.
\end{thm}

\begin{proof}
It is straightforward.
\end{proof}

A \emph{left ideal} of a skew left brace $A$ is a
subgroup $I$ of the additive group of $A$ such that
$\lambda_a(I)\subseteq I$ for all $a\in A$. It is not hard to prove
that a left ideal is a subgroup of the multiplicative group of the
skew left brace. An \emph{ideal} of $A$ is a left ideal $I$ of $A$
such that $a\circ I=I\circ a$ and $a+I=I+a$ for all $a\in A$.

\begin{defn}
    For a skew left brace $A$ let
\[
    \Fix(A)=\{a\in A:\lambda_x(a)=a\text{ for all $x\in A$}\}.
\]
\end{defn}
\begin{pro}
    Let $A$ be a skew left brace. Then $\Fix(A)$ is a left ideal of $A$.
\end{pro}
\begin{proof}
    A routine calculation proves that $\Fix(A)$ is a subgroup of the additive
    group of $A$. Clearly $\lambda_x(\Fix(A))\subseteq \Fix(A)$ for all $x\in A$.
\end{proof}
The following example shows that in general $\Fix(A)$ is not an ideal:
\begin{exa}
Consider the semidirect product $A=\Z/(3)\rtimes \Z/(2)$ of the
trivial braces $\Z/(3)$ and $\Z/(2)$
via the non-trivial action of $\Z/(2)$ over $\Z/(3)$.
We have
$$\lambda_{(x,y)}(a,b)=(x,y)(a,b)-(x,y)=(x+(-1)^ya,y+b)-(x,y)=((-1)^ya,b).$$
Hence $\Fix(A)=\{ (0,b)\mid b\in \Z/(2)\}$. Clearly $\Fix(A)$ is not a
normal subgroup of the multiplicative group of $A$. Thus $\Fix(A)$ is
not an ideal of $A$.
\end{exa}

If $X$ and $Y$ are subsets of a skew left brace $A$, $X*Y$ is the additive
subgroup of $A$ generated by elements of the form $x*y$, $x\in X$ and $y\in Y$, i.e.
\[
X*Y=\langle x*y:x\in X\,,y\in Y\rangle_+.
\]

\begin{lem}
    \label{lem:A*I}
    Let $A$ be a skew left brace. A subgroup $I$ of the additive group
    of $A$ is a left ideal of $A$ if and only if $A*I\subseteq I$.
\end{lem}

\begin{proof}
    Let $a\in A$ and $x\in I$. If $I$ is a
    left ideal, then $a*x=\lambda_a(x)-x\in I$. Conversely, if $A*I\subseteq
    I$, then $\lambda_a(x)=a*x+x\in I$.
\end{proof}

\begin{lem}
    \label{lem:I*A}
    Let $A$ be a skew left brace. A normal subgroup $I$ of the additive group of $A$
    is an ideal of $A$ if and only $\lambda_a(I)\subseteq I$ for all $a\in A$ and
    $I*A\subseteq I$.
\end{lem}

\begin{proof}
    Let $x\in I$ and $a\in A$.  Assume first that $I$ is invariant under the
    action of $\lambda$ and that $I*A\subseteq I$. Then
    \begin{equation}
        \label{eq:trick_lambda}
        \begin{aligned}
        a\circ x\circ a' &=a+\lambda_a(x\circ a')\\
        &=a+\lambda_a(x+\lambda_x(a'))
        =a+\lambda_a(x)+\lambda_a\lambda_x(a')+a-a\\
        &=a+\lambda_a(x+\lambda_x(a')-a')-a
        =a+\lambda_a(x+x*a')-a.
    \end{aligned}
    \end{equation}
    and hence $I$ is an ideal.

    Conversely, assume that $I$ is an ideal. Then $I*A\subseteq I$ since
    \begin{align*}
        x*a&=-x+x\circ a-a\\
        &=-x+a\circ(a'\circ x\circ a)-a
        =-x+a+\lambda_a(a'\circ x\circ a)-a\in I.
    \end{align*}
    This completes the proof.
\end{proof}

The socle of a skew left brace $A$ is defined as
\[
        \Soc (A)=\{ x\in A\colon x\circ a=x+a \text{ and } x+a=a+x, \text{ for all } a\in A\}.
\]
Clearly $\Soc(A)=\ker(\lambda)\cap Z(A,+)$. In \cite[Lemma~2.5]{MR3647970} it
is proved that $\Soc(A)$ is an ideal of $A$.
\begin{lem}
    \label{lem:socle}
    Let $A$ be a skew left brace and $a\in\Soc(A)$. Then $b+b\circ a=b\circ a+b$ and
    $\lambda_b(a)=b\circ a\circ b'$ for all $b\in A$.
\end{lem}
\begin{proof}
    Let $b\in A$. Since
    $b'\circ (b\circ a+b)=a-b'$ and
    $b'\circ (b+b\circ a)=-b'+a$, the first claim follows since
    $a\in Z(A,+)$.
    To prove the second claim one computes:
    \[
    b\circ a\circ b'=b\circ (a\circ b')=b\circ (a+b')=b\circ a-b=-b+b\circ
    a=\lambda_b(a),
    \]
    and the lemma is proved.
\end{proof}

\section{Left and right nilpotent skew left braces}
\label{nilpotency}

Let $A$ be a skew left brace. Following~\cite{MR2278047} one defines
$A^{(1)}=A$ and for $n\geq1$
\begin{align*}
    & A^{(n+1)}=A^{(n)}*A=\langle x*a: x\in A^{(n)},\,a\in A\rangle_+,
\end{align*}
where $\langle X\rangle_+$ denotes the subgroup of the additive group of $A$
generated by the subset $X$.  The series $A^{(1)}\supseteq A^{(2)}\supseteq
A^{(3)}\supseteq\cdots\supseteq A^{(n)}\cdots$ is the \emph{right series} of
$A$.

\begin{pro}
    \label{pro:right_series}
        Let $A$ be a skew left brace. Each $A^{(n)}$ is an ideal of $A$.
\end{pro}

\begin{proof}
    We want to prove that for each $n\in\N$, $A^{(n)}$ is a normal subgroup of
    $(A,+)$, that $\lambda_a(A^{(n)})\subseteq A^{(n)}$ for all $a\in A$ and
    that $A^{(n)}$ is a normal subgroup of $(A,\circ)$. We proceed by induction on $n$.
    The case $n=1$ is trivial. We assume that the claim is true for some $n\geq1$.  We first prove that
    $A^{(n+1)}$ is a normal subgroup of $(A,+)$. Let $a,b\in A$ and $x\in
    A^{(n)}$. Then $a+x*b-a\in A^{(n+1)}$ since
    \begin{align*}
        a+x*b-a&=a+\lambda_x(b)-b-a\\
        &=a+\lambda_x(b)-(a+b)
        =a+\lambda_x(-a+a+b)-(a+b)\\
        &=a+\lambda_x(-a)+\lambda_x(a+b)-(a+b)
        =-x*a+x*(a+b).
    \end{align*}
    Now we prove that $A^{(n+1)}$ is an ideal.
    Let $a,b\in A$ and $x\in A^{(n)}$. Then
\begin{equation}
    \label{eq:another_trick}
    \begin{aligned}
        \lambda_a(x*b)&=\lambda_a(\lambda_x(b)-b)
        =\lambda_a\lambda_x(b)-\lambda_a(b)\\
        &=\lambda_{a\circ x\circ a'}(\lambda_a(b))-\lambda_a(b)
        =(a\circ x\circ a')*\lambda_a(b)\in A^{(n+1)}
    \end{aligned}
\end{equation}
since $a\circ x\circ a'\in A^{(n)}$ by the inductive hypothesis.
From this it immediately follows that
$\lambda_a(A^{(n+1)})\subseteq A^{(n+1)}$.  Now let $y\in
A^{(n+1)}$. By using~\eqref{eq:trick_lambda} one obtains that
$a\circ y\circ a' =a+\lambda_a(y+y*a')-a\in A^{(n+1)}$. Thus
the result follows by induction.
\end{proof}

Let $A$ be a skew left brace. Following~\cite{MR2278047} one defines
$A^1=A$ and for $n\geq1$
\begin{align*}
    & A^{n+1}=A*A^{n}=\langle a*x: a\in A,\,x\in A^{n}\rangle_+.
\end{align*}
The series $
A^1\supseteq A^2\supseteq A^3\supseteq\cdots\supseteq A^n\supseteq\cdots$
is the \emph{left series} of $A$.

\begin{pro}
    \label{pro:left_series}
    Let $A$ be a skew left brace. Each $A^{n}$ is a left ideal of $A$.
\end{pro}

\begin{proof}
    We proceed by induction on $n$. The case $n=1$ is trivial, so we may assume
    that the result is true for some $n\geq1$. Let $a,b\in A$ and $x\in A^n$.
    By the inductive hypothesis, $\lambda_a(x)\in A^n$ and hence
    \[
        \lambda_a(b*x)=(a\circ b\circ a')*\lambda_a(x)\in A^{n+1},
    \]
where the equality follows by~\eqref{eq:another_trick}. This implies
that $\lambda_a(A^{n+1})\subseteq A^{n+1}$. Thus the result
follows by induction.
\end{proof}

The second term of the left series is particularly important:
\begin{pro}
 Let $A$ be a skew left brace. Then $A^2$ is the smallest ideal of $A$
 such that $A/A^2$ is a trivial skew left brace.
\end{pro}
\begin{proof}
    Since $A^2=A^{(2)}$, $A^2$ is an ideal by Proposition~\ref{pro:right_series}.
Let $I$ be an ideal of $A$ and $\pi\colon A\to A/I$ be the canonical map. Then
$A/I$ is trivial as a skew left brace if and only if $\lambda_a(b)-b\in I$ for all
$a,b\in A$. Since this condition is equivalent to $A^2\subseteq I$, the claim
follows.
\end{proof}

\begin{defn}
    A skew left brace $A$ is said to be \emph{right nilpotent}
    if $A^{(m)}=0$ for some $m\geq1$.
\end{defn}

\begin{lem}
    \label{lem:right_nilpotent:quotient}
    Let $f\colon A\to B$ be a surjective homomorphism of skew left braces. Then
    $f(A^{(k)})=B^{(k)}$ for all $k$. In particular, if $A$ is right nilpotent,
    then $B$ is right nilpotent.
\end{lem}
\begin{proof}
    We proceed by induction on $k$. The case $k=1$ is trivial. Let us assume that the result is valid for some $k\geq1$. Since
    $f(A^{(k)})=B^{(k)}$,
    \[
    f(A^{(k+1)})=f(A^{(k)}*A)=f(A^{(k)})*f(A)=B^{(k)}*B=B^{(k+1)}.
    \]
    From this the second claim follows.
\end{proof}
\begin{lem}
    \label{lem:right_nilpotent:sub}
    Let $A$ be a right nilpotent skew left brace and $B\subseteq A$ be a sub
    skew left brace. Then $B$ is right nilpotent.
\end{lem}
\begin{proof}
    By induction, $B^{(k)}\subseteq A^{(k)}$ for all $k$. Hence the claim
    follows.
\end{proof}

\begin{lem}
    \label{lem:right_nilpotent:x}
    Let $A_1,\dots,A_k$ be right nilpotent skew left braces. Then the direct product
    $A_1\times\cdots\times A_k$ is right nilpotent.
\end{lem}
\begin{proof}
    It is enough to prove the lemma in the case where $k=2$. This case is
    trivial since $(a,b)*(c,d)=(a*c,b*d)$.
\end{proof}

\begin{thm}
    \label{thm:IcapSoc}
    Let $A$ be a right nilpotent skew left brace of nilpotent type
    and $I$ be a non-zero ideal of $A$. Then $I\cap\Soc(A)\ne0$.
\end{thm}
\begin{proof}
    Since $(A,+)$ is nilpotent and each $I\cap A^{(k)}$ is a normal subgroup of $(A,+)$, it follows
    from~\cite[5.2.1]{MR1357169} that $I\cap A^{(k)}\cap Z(A,+)\ne0$ whenever
    $I\cap A^{(k)}\ne0$.
    Let
    $m=\max\{k\in\N:I\cap A^{(k)}\cap Z(A,+)\ne 0\}$.
    Since
    \[
        (I\cap A^{(m)}\cap Z(A,+))*A\subseteq I\cap (A^{(m)}*A)=I\cap A^{(m+1)}=0,
    \]
    it follows that $I\cap A^{(m)}\cap Z(A,+)\subseteq I\cap \Soc(A)$.
\end{proof}

\begin{cor}
    \label{cor:nonzero_socle}
    Let $A$ be a non-zero right nilpotent skew left brace of nilpotent type.
    Then $\Soc(A)\ne0$.
\end{cor}
\begin{proof}
    It follows directly from Theorem~\ref{thm:IcapSoc}
\end{proof}
\begin{cor}
    Let $A$ be a right nilpotent skew left brace of nilpotent type and $I$ be a
    minimal ideal of $A$. Then $I\subseteq\Soc(A)$.
\end{cor}
\begin{proof}
    Since $I\cap\Soc(A)$ is a non-zero ideal of $A$ by Theorem~\ref{thm:IcapSoc},
    $I\cap\Soc(A)=I$ by the minimality of $I$.
\end{proof}

\begin{defn}
    Let $A$ be a skew left brace. A \emph{$s$-series} of $A$ is a sequence
    \[
        A=I_0\supseteq I_1\supseteq I_2\supseteq\cdots\supseteq I_n=0
    \]
    of ideals of $A$ such that $I_{j-1}/I_j\subseteq \Soc(A/I_j)$ for each
    $j\in\{1,\dots,n\}$.
\end{defn}

\begin{rem}
Let $A$ be a left brace.  Rump in \cite{MR2278047} defined the socle
series $\Soc_n(A)$ of $A$ as follows: $\Soc_0(A)=0$ and, for $n\geq
1$,
\[
\Soc_n(A)=\{ x\in A\colon x*y\in
\Soc_{n-1}(A) \}.
\]
There are examples of nonzero left braces $A$ such that
$\Soc_n(A)=0$ for all positive integers $n$.
\end{rem}

\begin{defn}\label{socn}
Let $A$ be a skew left brace. We define
$\Soc_0(A)=0$ and, for $n\geq 1$, $\Soc_n(A)$ is the ideal of $A$
containing $\Soc_{n-1}(A)$ such that
\[
    Soc_n(A)/\Soc_{n-1}(A)=\Soc(A/\Soc_{n-1}(A)).
\]
\end{defn}

Note that this definition coincides with the
definition of Rump for left braces.

\begin{lem}
    \label{lem:socle_series}
    Let $A$ be a skew left brace and let $A=I_0\supseteq I_1\supseteq
    I_2\supseteq\cdots\supseteq I_n=0$ be a $s$-series for $A$. Then
    $A^{(i+1)}\subseteq I_i$ for all $i$.
\end{lem}
\begin{proof}
    We proceed by induction on $i$. The case $i=0$ is trivial, so let us assume
    that the result holds for some $i\geq 0$. Let $\pi\colon A\to A/I_{i+1}$ be the canonical map.
    Since $\pi(I_{i})\subseteq\Soc(A/I_{i+1})$,
    $\pi(I_{i}*A)=\pi(I_{i})*\pi(A)=0$ and hence $I_{i}*A\subseteq I_{i+1}$.
    The inductive hypothesis then implies that
    $A^{(i+2)}=A^{(i+1)}*A\subseteq I_{i}*A\subseteq I_{i+1}$. Thus
the result follows by induction.
\end{proof}

\begin{lem}
    \label{lem:socn}
Let $A$ be a skew left brace. Then $A$ admits a
$s$-series if and only if there exists a positive integer $n$ such
that $A=\Soc_n(A)$.
\end{lem}

\begin{proof}
    Suppose that there exists a positive integer $n$ such
that $A=\Soc_n(A)$. Then \[A=\Soc_n(A)\supseteq
\Soc_{n-1}(A)\supseteq\cdots\supseteq \Soc_0(A)=0,\] is a
$s$-series.

Conversely, suppose that $A$ admits a $s$-series.
Let
$$A=I_0\supseteq I_1\supseteq I_2\supseteq\cdots\supseteq I_n=0$$
be a $s$-series of $A$. We shall prove that $I_{n-j}\subseteq
\Soc_{j}(A)$ by induction on $j$. For $j=0$, $I_n=0=\Soc_0(A)$.
Suppose that $j>0$ and $I_{n-j+1}\subseteq \Soc_{j-1}(A)$. Since
$I_{n-j}/I_{n-j+1}\subseteq \Soc(A/I_{n-j+1})$, $I_{n-j}*A\subseteq
I_{n-j+1}\subseteq \Soc_{j-1}(A)$, by the induction hypothesis.
Furthermore, for all $x\in A$ and all $y\in I_{n-j}$, $x+y-x-y\in
I_{n-j+1}\subseteq \Soc_{j-1}(A)$. Therefore
$I_{n-j}\subseteq\Soc_j(A)$. Hence $A=I_0=\Soc_n(A)$ and the result
follows.
\end{proof}

\begin{lem}
    \label{lem:right_nilpotent}
    A skew left brace of nilpotent type is right nilpotent if and only if it admits a $s$-series.
\end{lem}
\begin{proof}
    Let $A$ be a skew left brace {of nilpotent type}. If $A$ admits a $s$-series, then $A$ is right nilpotent by
    Lemma~\ref{lem:socle_series}.

Conversely, suppose that $A$ is right nilpotent.
There exists a positive integer such that $A^{(m)}=0$.
    We shall prove that $A$ admits a $s$-series by induction on $m$. For $m=1$, $A=A^{(1)}=0$ is a $s$-series. Suppose
    that $m>1$ and that the result is true for $m-1$. Consider
    $\bar A=A/A^{(m-1)}$. Since $\bar A^{(m-1)}=0$, by the induction
    hypothesis $\bar A$ admits a $s$-series. Thus there is a
    sequence
    \[
        A=I_0\supseteq I_1\supseteq I_2\supseteq\cdots\supseteq I_n=A^{(m-1)}
    \]
    of ideals of $A$ such that $I_{j-1}/I_j\subseteq \Soc(A/I_j)$ for each
    $j\in\{1,\dots,n\}$. Since $A^{(m)}=0$, we have that $A^{(m-1)}\subseteq
    \ker(\lambda)$. Since $A$ is of nilpotent type, there exists a
    positive integer $s$ such that $\gamma^+_s(A)=0$, where
    $\gamma^+_i(A)$ denotes the lower central series of the additive
    group of $A$, that is $\gamma^+_1(A)=A$ and $\gamma^+_{i+1}(A)=[A,\gamma^+_i(A)]_+$, for all positive integers $i$.
    Let $I_{n+j-1}=A^{(m-1)}\cap \gamma^+_{j}(A)$ for $j=1,\dots ,s$.
    Note that $I_{n+j-1}$ is a normal subgroup of the additive group
    of $A$ invariant by $\lambda_x$, for all $x\in A$, and
    $I_{n+j-1}* A=0$, for all $j=1,\dots , s$, because $A^{(m-1)}\subseteq
    \ker(\lambda)$. By Lemma~\ref{lem:I*A}, $I_{n+j-1}$ is an ideal
    of $A$, for all $j=1,\dots , s$.
    Note that $I_{n+j-1}/I_{n+j}\subseteq Z(A/I_{n+j},+)$, for all $j=1,\dots ,s-1$. Therefore, since
    $I_{n+j-1}\subseteq\ker(\lambda)$, we have that $I_{n+j-1}/I_{n+j}\subseteq \Soc(A/I_{n+j})$, for all $j=1,\dots ,s-1$.
    Hence
    \[
        A=I_0\supseteq I_1\supseteq I_2\supseteq\cdots\supseteq
        I_n=A^{(m-1)}\supseteq I_{n+1}\supseteq\cdots\supseteq
        I_{n+s-1}=0
    \]
    is a $s$-series of $A$, and the result follows by induction.
\end{proof}

\begin{pro}
    \label{thm:A/Soc}
    Let $A$ be a skew left brace such that $A/\Soc(A)$ is right nilpotent. Then $A$
    is right nilpotent.
\end{pro}
\begin{proof}
    Note that $(A/\Soc(A))^{(k)}=0$ if and only if $A^{(k)}\subseteq\Soc(A)$ by
    the definition of the factor brace. Then $A^{(k+1)}=A^{(k)}*A\subseteq\Soc
    (A)*A=0$ as required.
\end{proof}

\begin{pro}
    Let $I$ be an ideal of a skew left brace $A$ such that $I\cap A^2=0$. Then $I$
    is a trivial skew left brace.
\end{pro}
\begin{proof}
    Since $I*A\subseteq I\cap A^2=0$, $I\subseteq\ker\lambda$. From this the
    claim follows.
\end{proof}

A skew left brace $A$ has \emph{finite multipermutation level} if
the sequence $S_n$ defined as $S_1=A$ and $S_{n+1}=S_n/\Soc(S_n)$
for $n\geq1$, reaches zero.

\begin{pro}\label{newmulti}
Let $A$ be a skew left brace. Then $A$ has finite
multipermutation level if and only if $A$ admits a $s$-series.
\end{pro}

\begin{proof}
Let $S_1=A$ and $S_{n+1}=S_n/\Soc(S_n)$, for
$n\geq1$. We shall prove that $S_{n+1}\simeq A/\Soc_n(A)$, by
induction on $n$. For $n=0$, it is clear since $\Soc_0(A)=0$.
Suppose that $n>0$ and the result is true for $n-1$. Hence, by the
induction hypothesis,
\begin{align*}S_n&=S_{n-1}/\Soc(S_{n-1})
\simeq (A/\Soc_{n-1}(A))/\Soc(A/\Soc_{n-1}(A))\\
&=(A/\Soc_{n-1}(A))/(\Soc_n(A)/\Soc_{n-1}(A))\simeq
A/\Soc_n(A).\end{align*} Therefore $S_n=0$ if and only if
$A=\Soc_n(A)$. Now the result follows from Lemma~\ref{lem:socn}.
\end{proof}

\begin{thm}
\label{thm:mpl&right_nilpotent}
Let $A$ be a skew
left brace. Then $A$ has finite multipermutation level, if and only
if  $A$ is right nilpotent and $(A,+)$ is nilpotent.
\end{thm}

\begin{proof}
Suppose that $A$ has finite multipermutation level.
We proceed by induction on the multipermutation level $n$. The case
$n=1$ is trivial.  Let $A$ be a skew left brace of
    finite multipermutation level $n+1$. Since $A/\Soc(A)$ has multipermutation
    level $n$, the inductive hypothesis implies that $(A/\Soc(A))^{(m)}=0$ for
    some $m$ and $(A/\Soc(A),+)$ is nilpotent. This implies that
    $A^{(m)}\subseteq\Soc(A)$ and hence $A^{(m+1)}=0$, furthermore, since
    $\Soc(A)$ is central in $(A,+)$, we have that $(A,+)$ is nilpotent.
Conversely, suppose that $A$ is right nilpotent and $(A,+)$ is
nilpotent.  By Lemma~\ref{lem:right_nilpotent}, $A$ admits a
$s$-series. Thus the result follows by Proposition~\ref{newmulti}.
\end{proof}

The following example shows that the assumption on the nilpotency of
the additive group of the skew left brace is needed for
Theorem~\ref{thm:mpl&right_nilpotent}.

\begin{exa}\label{ex:trivial}
Let $A$ be a non-zero skew left brace such that $a\circ b=a+b$ for
all $a,b\in A$. Then $A^{(2)}=0$, thus $A$ is right nilpotent. But
if $Z(A,+)=0$, then $\Soc(A)=0$ and $A$ does not have finite
multipermutation level. For example, we can take $(A,+)=(A,\circ)$
any non-abelian simple group.
\end{exa}

\begin{defn}
    A skew left brace $A$ is said to be \emph{left nilpotent} if $A^{m}=0$ for some
    $m\geq1$.
\end{defn}

\begin{lem}
    Let $f\colon A\to B$ be a surjective homomorphism of skew left braces. Then
    $f(A^{k})=B^{k}$ for all $k$. In particular, if $A$ is left nilpotent,
    then $B$ is left nilpotent.
\end{lem}
\begin{proof}
    It is similar to the proof of Lemma~\ref{lem:right_nilpotent:quotient}.
\end{proof}
\begin{lem}
    Let $A$ be a left nilpotent skew left brace and $B\subseteq A$ be a sub skew left
    brace. Then $B$ is left nilpotent.
\end{lem}
\begin{proof}
    It is similar to the proof of Lemma~\ref{lem:right_nilpotent:sub}.
\end{proof}
\begin{lem}
    \label{lem:left_nilpotent:x}
    Let $A_1,\dots,A_k$ be left nilpotent skew left braces. Then the direct product
    $A_1\times\cdots\times A_k$ is left nilpotent.
\end{lem}
\begin{proof}
    It is similar to the proof of Lemma~\ref{lem:right_nilpotent:x}.
\end{proof}
\begin{pro}
    \label{thm:Fix}
    Let $A$ be a left nilpotent skew left brace and $I$ be a non-zero left ideal of
    $A$. Then $I\cap\Fix(A)\ne0$.
\end{pro}

\begin{proof}
    Let $m=\max\{k:I\cap A^k\ne0\}$. Since
    $A*(I\cap A^m)\subseteq I\cap A^{m+1}=0$,
    it follows that there exists a non-zero $x\in I\cap A^m$ such that $a*x=0$
    for all $a\in A$. Thus $0\ne x\in\Fix(A)\cap I$.
\end{proof}

\begin{cor}
    Let $A$ be a non-zero skew left brace. If $A$ is left nilpotent, then $\Fix(A)\ne0$.
\end{cor}
\begin{proof}
    Apply Proposition~\ref{thm:Fix} with $I=A$.
\end{proof}

Another important series of ideals was defined
in~\cite{MR3814340} for braces. Let $A$ be a skew left brace.
Let $A^{[1]}=A$ and for $n\geq 1$ let $A^{[n+1]}$ be the
additive subgroup of $A$ generated by elements from
$\{A^{[i]}*A^{[n+1-i]}:1\leq i\leq n\}$, i.e.
\[
    A^{[n+1]}=\left\langle \bigcup_{i=1}^n A^{[i]}*A^{[n+1-i]}\right\rangle_+
\]
for all $n\geq2$. One easily proves by induction that $A^{[1]}\supseteq
A^{[2]}\supseteq\cdots$.

\begin{pro}
    \label{pro:Smoktunowicz}
    Let $A$ be a skew left brace. Each $A^{[n]}$ is a left ideal of $A$.
\end{pro}

\begin{proof}
    Each $A^{[n]}$ is a subgroup of $(A,+)$. Since
    $A*A^{[n]}\subseteq A^{[n+1]}\subseteq A^{[n]}$, the claim follows from
    Lemma~\ref{lem:A*I}.
\end{proof}

    We will show that there exists a left brace $A$ such that $A^{[n]}=A^{[n+1]}$
    are non-zero for some positive integer $n$ and $A^{[n+2]}=0$.

\begin{exa}
    \label{exa:funny}
    Let
    \begin{align*}
        &G=\langle r,s:r^8=s^2=1,\,srs=r^7\rangle\simeq\D_{16},\\
        &X=\langle a,b:8a=2b=0,\,a+b=b+a\rangle\simeq \Z/(8)\times \Z/(2).
    \end{align*}
    The group $G$ acts by automorphisms on $X$ via
    \[
        r\cdot a=a+b,\quad r\cdot b=4a+b,\quad
        s\cdot a=3a, \quad s\cdot b=4a+b.
    \]
    A direct calculation shows that the map $\pi\colon G\to X$ given by
    \begin{align*}
        1 &\mapsto 0, & r&\mapsto a, & r^2&\mapsto 2a+b, & r^3&\mapsto 7a+b,\\
        r^4 &\mapsto 4a, & r^5&\mapsto 5a, & r^6&\mapsto 6a+b, & r^7&\mapsto 3a+b,\\
        rs &\mapsto 6a, & r^2s&\mapsto 7a, & r^3s&\mapsto b, &r^4s&\mapsto 5a+b,\\
        r^5s &\mapsto 2a, & r^6s&\mapsto 3a, &r^7s&\mapsto 4a+b,&s&\mapsto a+b,
    \end{align*}
    is a bijective $1$-cocycle. Therefore there exists a left brace $A$ with
    additive group isomorphic to $X$ and multiplicative group isomorphic to
    $G$. The addition of $A$ is that of $X$ and the multiplication is given by
    \[
        x\circ y=\pi(\pi^{-1}(x)\pi^{-1}(y)),\quad x,y\in X.
    \]
    Since
    \begin{align*}
        a*a&=-a+a\circ a-a=-a+(2a+b)-a=b,\\
        (5a+b)*a&=-(5a+b)+(5a+b)\circ a-a=-(5a+b)+b-a=2a,
    \end{align*}
    it follows that $A^{[2]}$ contains $\langle
    2a,b\rangle_+=\{0,2a,4a,6a,b,2a+b,4a+b,6a+b\}$, the additive subgroup of
    $(A,+)$ generated by $2a$ and $b$. Therefore $A^{[2]}=\langle 2a,b\rangle_+$
    since $A^{[2]}\ne A$ (this can be proved by hand or using Theorem~\ref{thm:left_nilpotent=nilpotent}). Routine calculations prove that
    \begin{align*}
        A^{[3]}=\{0,2a+b,4a,6a+b\}, && A^{[4]}=A^{[5]}=\{0,4a\}, && A^{[6]}=\{0\}.
    \end{align*}
\end{exa}

The relation between the sequence of the $A^{[n]}$ and the left and
right series is given in the following theorem.

\begin{thm}
    \label{thm:equivalence}
    Let $A$ be a skew left brace. The following statements are equivalent:
    \begin{enumerate}
        \item $A^{[\alpha]}=0$ for some $\alpha\in\N$.
        \item $A^{(\beta)}=0$ and $A^\gamma=0$ for some $\beta,\gamma\in\N$.
    \end{enumerate}
\end{thm}
\begin{proof}
    To prove that $(1)\implies(2)$ one proves by induction that $A^{n}\subseteq
    A^{[n]}$ and $A^{(n)}\subseteq A^{[n]}$
    for all positive integer $n$.
    Let us prove that
    $(2)\implies(1)$. We proceed by induction on $\beta$.  If
    $\beta\in\{1,2\}$, then $0=A^{(2)}=A^2=A^{[2]}$ and the result is true. Fix
    $\beta\in\N$ and suppose that the result holds for this $\beta$, so for
    every $\gamma$ there exists $\alpha=\alpha(\gamma)$ depending on $\gamma$
    such that $A^{[\alpha]}=0$.  We need to show that $A^{(\beta+1)}=0$ and
    $A^{\gamma}=0$ imply that $A^{[n]}$ for some $n$. Let $n>\alpha(\gamma)$.
    Every element of $A^{[n]}$ is a sum of elements from $A^{[i]}*A^{[j]}$,
    where $i+j=n$ and $1\leq i\leq n-1$. Note that if $\alpha(\gamma)\leq i\leq
    n-1$, $a_i\in A^{[i]}$ and $a_{n-i}\in A^{[n-i]}$, then by the inductive
    hypothesis applied to the quotient $A/A^{(\beta)}$, $a_i\in A^{(\beta)}$ and thus $a_i*a_{n-i}\in A^{(\beta
    +1)}=0$. Hence we may assume that the elements of $A^{[n]}$ are sums of
    elements from $A^{[i]}*A^{[j]}$ for $1\leq i<\alpha(\gamma)$ and $j\geq
    n-\alpha(\gamma)$ such that $i+j=n$. Then
    \[
    A^{[n]}\subseteq A*A^{[n-\alpha(\gamma)]}\subseteq A^2.
    \]
    Applying the same argument for $n'=n-\alpha(\gamma)$ we obtain that
    $A^{[n']}\subseteq A*A^{[n'-\alpha(\gamma)]}$ provided that
    $n'>\alpha(\gamma)$. Therefore
    \[
        A^{[n]}\subseteq A*(A*A^{[n-2\alpha(\gamma)]})\subseteq A^3
    \]
    provided that $n>2\alpha(\gamma)$. Continuing in this way we obtain that
    $A^{[n]}\subseteq A^{k}$
    provided that $n>(k-1)\alpha(\gamma)$.
    Then it follows that $A^{[(\gamma-1)\alpha(\gamma)+1]}\subseteq
    A^{\gamma}=0$.
\end{proof}

\begin{defn}
A skew left brace is said to be \emph{left nil} if for every
$a\in A$ we have $a*(a*(a*\cdots))=0$ (for sufficiently large number
of brackets in this equation).
\end{defn}


It is known that every finite left nil left brace is left
nilpotent, see~\cite{MR3765444}.  Right nil skew left braces
are defined in a similar fashion.
\begin{defn}
A skew left brace $A$ is said to be \emph{strongly nilpotent}
if there is a positive integer $n$ such that $A^{[n]}=0$.
\end{defn}
\begin{defn}
    A skew left brace $A$ is said to be \emph{strongly nil} if for every
    $a\in A$ there is a positive integer $n=n(a)$ such that any $*$-product of
    $n$ copies of $a$ is zero.
\end{defn}

We do not know the answer to the following questions:
\begin{question}
    \label{question:rightnil=>rightnilp}
    Let $A$ be a finite right nil skew left brace.  Is $A$ right
    nilpotent?
\end{question}
\begin{question}
    \label{question:stronglynil=>stronglynilp}
    Let $A$ be a finite strongly nil skew left brace.
    Is $A$ strongly nilpotent?
\end{question}

%
%

\section{Perfect skew left braces}
\label{perfect}

A skew left brace $A$ is said to be perfect if $A^2 = A$.  Let $G$
be a perfect group, that is $G=[G,G]$. By Gr\"{u}n's Lemma (see
\cite[page 3]{grun}), $Z(G/Z(G))=\{ 1\}$. Let $B$ be the skew left brace with multiplicative
group $G$ and addition defined by $a+b=ba$, for all $a,b\in G$. In $B$ we have
$a*b=-a+ab-b=b^{-1}aba^{-1}$, for all $a,b\in G$. Hence the multiplicative
group of $B*B$ is $[G,G]$. Note also that $x\in \Soc(B)$ if and only if
$1=a^{-1}xax^{-1}$, for all $a\in G$. Thus $\Soc(B)=Z(G)$. Since $Z(G/Z(G))=\{
1\}$, we have that $\Soc(B/Soc(B)) =\{1\}$.

Thus the following question appears to be natural.

\begin{question}\label{newLeandro} Let $A$ be a perfect skew left brace. Is
$\Soc(A/Soc(A)) = 0$?
\end{question}

We shall see that the answer is negative.

Let $B_1$ and $B_2$ be skew left braces. Recall that the wreath product
$B_2\wr B_1$ of the left braces $B_2$ and $B_1$ is a left brace
which is the semidirect product of left braces $H_2\rtimes B_1$,
where $H_2=\{ f\colon B_1\longrightarrow B_2\mid |\{g\in B_1\mid
f(g)\neq 0\}|<\infty\}$ is a left brace with the operations
$(f_1\circ f_2)(g)=f_1(g)\circ f_2(g)$ and
$(f_1+f_2)(g)=f_1(g)+f_2(g)$, for all $f_1,f_2\in B_2$ and $g\in
B_1$, and the action of $(B_1,\circ)$ on $H_2$ is given by the
homomorphism $\sigma\colon (B_1,\circ)\longrightarrow \Aut(H_2,
+,\circ)$ defined by $\sigma(g)(f)(x)=f(g'\circ x)$, for all $g,x\in
B_1$ and $f\in H_2$. Recall that
the operations of $H_2\rtimes B_1$ are
\begin{align*}
    &(h_2,b_1)\circ (k_2,c_1)=(h_2\circ b_1(k_2),b_1\circ c_1),\\
    &(h_2,b_1)+(k_2,c_1)=(h_2+k_2,b_1+c_1),
\end{align*}
where we denote $\sigma(b_1)(k_2)$ simply by $b_1(k_2)$.

The wreath product of left braces appears in \cite[Corollary 1]{MR3177933} (see
also \cite[Section 4]{BCJO18}). This construction also works for skew left
braces (see \cite[Corollary 2.39]{MR3763907}).

\begin{thm}\label{thm:perfect}
    Let $B$ be a finite perfect left skew left brace. Let $p$ be an odd prime
    non-divisor of the order of $B$. Let $T=\Z/(p)$ be the trivial left
    brace of order $p$. Then the subbrace $W\rtimes B$ of $T\wr B$,
    where $W=\{ f\colon B\longrightarrow T\mid \sum_{b\in B}f(b)=0\}$,
    is perfect and $\Soc(W\rtimes B)=W\times\{ 0\}$.
\end{thm}

\begin{proof}
Note that
\begin{equation}\label{eq:wreath}(h_1,b_1)*(h_2,b_2)=(-h_1+h_1\circ
b_1(h_2)-h_2,b_1*b_2),\end{equation} for all $h_1,h_2\in W$ and all
$b_1,b_2\in B$. In particular,
$$\{0\}\times B=\{0\}\times (B*B)=(\{0\}\times B)*(\{0\}\times B)\subseteq (W\rtimes B)*(W\rtimes B).$$
Let $f_{b_1}\colon B\longrightarrow T$ be the function defined by
$f_{b_1}(b_2)=\delta_{b_1,b_2}$ (the Kronecker delta), for all
$b_1,b_2\in B$. Note that $\{f_{b}-f_0\mid b\in B\}$ generates the
additive group of $W$. Since $p$ is not a divisor of $|B|$, there
exists $|B|^{-1}\in \Z/(p)=T$, and $f_0-|B|^{-1}\sum_{b\in B}f_b\in
W$ (note that the additive group of $W$ is a vector space over
$\Z/(p)$). Now we have
$$(0,b_1)*(f_0-|B|^{-1}\sum_{b\in B}f_b,0)=(b_1(f_0-|B|^{-1}\sum_{b\in B}f_b)-(f_0-|B|^{-1}\sum_{b\in B}f_b),0)$$
and
\begin{eqnarray*}
    b_1(f_0-|B|^{-1}\sum_{b\in B}f_b)(b_2)&=&(f_0-|B|^{-1}\sum_{b\in B}f_b)(b_1'\circ b_2)\\
&=&f_0(b_1'\circ b_2)-|B|^{-1}\sum_{b\in B}f_b(b_1'\circ b_2)\\
&=&f_{b_1}(b_2)-|B|^{-1}\\
&=&f_{b_1}(b_2)-|B|^{-1}\sum_{b\in B}f_b(b_2)\\
&=&(f_{b_1}-|B|^{-1}\sum_{b\in B}f_b)(b_2)
\end{eqnarray*}
for all $b_1,b_2\in B$. Hence
$$(0,b_1)*(f_0-|B|^{-1}\sum_{b\in B}f_b,0)=(f_{b_1}-f_0,0),$$
for all $b_1\in B$. Thus $W\times\{0\}\subseteq (W\rtimes
B)*(W\rtimes B)$. Hence
$$W\rtimes B=W\times \{0\}+\{0\}\times B\subseteq (W\rtimes B)*(W\rtimes B).$$
Therefore $W\rtimes B$ is perfect.

Let $(h_1,b_1)\in \Soc(W\rtimes
B)$. Then, by (\ref{eq:wreath}), $h_1\circ b_1(h_2)=h_1+h_2$ and
$b_1*b_2=0$, for all $h_2\in W$ and all $b_2\in B$. Hence
$$h_1(b)+h_2(b)=h_1(b)\circ b_1(h_2)(b)=h_1(b)\circ h_2(b_1'\circ b)=h_1(b)+ h_2(b_1'\circ b),$$
and thus
$$h_2(b)=h_2(b_1'\circ b),$$
for all $h_2\in W$ and all $b\in B$. In particular,
$(f_{b_1}-f_0)(b_1)=(f_{b_1}-f_0)(0)$. Suppose that $b_1\neq 0$.
Then $1=(f_{b_1}-f_0)(b_1)=(f_{b_1}-f_0)(0)=-1$ in $\Z/(p)$, a
contradiction because $p$ is odd. Hence $b_1=0$. Note that by
(\ref{eq:wreath}),
$$(h_1,0)*(h_2,b_2)=(-h_1+h_1\circ
0(h_2)-h_2,b_1*0)=(-h_1+h_1\circ h_2-h_2,0)=(0,0),$$ for all
$h_1,h_2\in W$ and $b_2\in B$. Hence $\Soc(W\rtimes
B)=W\times\{0\}$, and the result follows.
\end{proof}

Note that in Theorem~\ref{thm:perfect}, if $B$ is a left brace, then
$W\rtimes B$ also is a left brace. The following result answers
Question~\ref{newLeandro} in the negative:
\begin{cor}
For every positive integer $n$, there exists a finite perfect left
brace $B$ such that $\Soc(B/\Soc_n(B))\neq \{0\}$.
\end{cor}

\begin{proof}
We shall prove the result by induction on $n$. For $n=1$, let $B_0$
be a finite simple non-trivial left brace (see \cite{MR3763276}).
Then by Theorem~\ref{thm:perfect}, there exists a perfect finite left
brace $B_1$ with non-zero socle. By Theorem~\ref{thm:perfect}, there
exists a perfect finite left brace $B_2$ with non-zero socle such
that $B_2/\Soc(B_2)\cong B_1$. Therefore
$\Soc(B_2/\Soc_1(B_2))\cong\Soc(B_1)\neq \{ 0\}$, and this proves
the result for $n=1$.

Suppose that $n>1$ and that there exists a perfect finite left brace
$B_{n}$ with $\Soc(B_{n}/\Soc_{n-1}(B_{n}))\neq\{0\}$. By
Theorem~\ref{thm:perfect}, there exists a perfect finite left brace
$B_{n+1}$ such that $\Soc(B_{n+1})\neq\{0\}$ and
$B_{n+1}/\Soc(B_{n+1})\cong B_{n}$. Hence
\begin{align*}
\Soc(B_{n+1}/\Soc_{n}(B_{n+1}))&\cong(\Soc((B_{n+1}/\Soc(B_{n+1}))/(\Soc_{n}(B_{n+1})/\Soc(B_{n+1}))\\
&\cong\Soc(B_{n}/\Soc_{n-1}(B_{n}))\neq\{0\},
\end{align*}
by the inductive hypothesis. By induction the result follows.
\end{proof}

\section{Skew braces of nilpotent type}
\label{nilpotent_type}

We first prove that if both groups of a finite skew left brace $A$ are
nilpotent, then $A$ can be decomposed as a direct product of skew left braces
of prime-power size. A similar result was proved by Byott in the context of
Hopf--Galois extensions, see~\cite[Theorem 1]{MR3030514}.

\begin{lem}\label{sum}
    Let $A$ be a skew left brace such that the additive
group is a direct sum of ideals $I_1,I_2$, that is $A=I_1+I_2$ and
$I_1\cap I_2=\{0\}$. Then the map $f:A\rightarrow I_1\times I_2$
defined by $f(a_1+a_2)=(a_1,a_2)$, for all $a_1\in I_1$ and $a_2\in
I_2$, is an isomorphism of skew left braces.
\end{lem}

\begin{proof}
Recall that the operations of the skew left brace $I_1\times I_2$
are defined componentwise. Clearly $f$ is an isomorphism of the
additive groups of $A$ and $I_1\times I_2$. Let $a_1\in I_1$ and
$a_2\in I_2$. Since $I_1$ and $I_2$ are ideals we have that
$$a_1+a_2-a_1-a_2, a_1*a_2, a_2*a_1\in I_1\cap I_2=\{ 0\},$$
thus $a_1+a_2=a_2+a_1$ and $a_1\circ a_2=a_1+a_2=a_2\circ a_1$.
Hence
\begin{eqnarray*} f((a_1+a_2)\circ (b_1+b_2))&=&f(a_1\circ
a_2\circ b_1\circ b_2)=
f(a_1\circ b_1\circ a_2\circ b_2)\\
&=&f(a_1\circ b_1 + a_2\circ b_2)=(a_1\circ b_1 , a_2\circ b_2)\\
&=&(a_1,a_2)\circ (b_1,b_2)=f(a_1+a_2)\circ f(b_1 +b_2),
\end{eqnarray*}
for all $a_1,b_1\in I_1$ and $a_2,b_2\in I_2$. Therefore, the result
follows.
\end{proof}

\begin{thm}
    \label{direct}
    Let $n$ be a positive integer. Let $A$ be a skew left brace such that the additive
group is a direct sum of ideals $I_1,\dots ,I_n$, that is every
element $a\in A$ is uniquely written as $a=a_1+\dots +a_n$, with
$a_j\in I_j$ for all $j$. Then the map $f:A\rightarrow
I_1\times\dots\times I_n$ defined by $f(a_1+\dots +a_n)=(a_1,\dots
,a_n)$, for all $a_j\in I_j$, is an isomorphism of skew left braces.
\end{thm}

\begin{proof}
We shall prove the result by induction on $n$. For $n=1$, it is
clear. Suppose that $n>1$ and that the result is true for $n-1$. Let
$A_1=I_1+\dots +I_{n-1}$. Then $A_1$ is an ideal of $A$ and $A$ is
the direct sum of the ideals $A_1$ and $I_n$. By Lemma~\ref{sum},
the map $f_1: A\rightarrow A_1\times I_n$ defined by
$f(a+a_n)=(a,a_n)$, for all $a\in A_1$ and $a_n\in I_n$, is an
isomorphism of skew left braces. By the induction hypothesis, the
map
\[
    f_2:A_1\rightarrow I_1\times \dots\times I_{n-1},
    \quad
    f_2(a_1+\dots +a_{n-1})=(a_1,\dots ,a_{n-1}),
\]
is an isomorphism of skew left braces. Therefore $f=(f_2\times
\id)\circ f_1:A\rightarrow I_1\times\dots\times I_n$ is an
isomorphism of skew left braces and $f(a_1+\dots +a_n)=(a_1,\dots
,a_n)$, for all $a_j\in I_j$. Therefore, the result follows by
induction.
\end{proof}

\begin{cor}
    \label{cor:product}
    Let $A$ be a finite skew left brace such that $(A,+)$ and $(A,\circ)$ are nilpotent.
Let $I_1,\dots ,I_n$ be the distinct Sylow subgroups of the additive
group of $A$. Then $I_1,\dots ,I_n$ are ideals of $A$ and the map
$f:A\rightarrow I_1\times\dots\times I_n$ defined by $f(a_1+\dots
+a_n)=(a_1,\dots ,a_n)$, for all $a_j\in I_j$, is an isomorphism of
skew left braces.
\end{cor}

\begin{proof}
Since $(A,+)$ is nilpotent, for every prime divisor $p$ of the order
of $A$, there is a unique Sylow $p$-subgroup $I$ of $(A,+)$. Hence $I$
is a normal subgroup of $(A,+)$, and $\lambda_a(b)\in I$ for all
$a\in A$ and $b\in B$. Thus $I$ is a left ideal of $A$ and thus it
is a Sylow $p$-subgroup of $(A,\circ)$. Since $(A,\circ)$ is
nilpotent, $I$ is the unique Sylow $p$-subgroup of $(A,\circ)$ and,
thus, it is normal in $(A,\circ)$. Therefore $I$ is an ideal of $A$.
Hence $I_1,\dots ,I_n$ are ideals of $A$ and clearly the additive
group of $A$ is the direct sum of $I_1,\dots ,I_n$. The result
follows by Theorem~\ref{direct}.
\end{proof}

Let $A$ be a skew left brace.  Let $G$ be the multiplicative group
of $A$ and $X$ be the additive group of $A$. Since $G$ acts on
$X$ by automorphisms, one forms the semidirect product
$\Gamma=X\rtimes G$ with multiplication
\[
    (x,g)(y,h)=(x+\lambda_g(y),g\circ h).
\]

Identifying each $g\in G$ with $(0,g)\in\Gamma$ and each $x\in X$
with $(x,0)\in\Gamma$, we see that
\begin{align*}
[g,x]&= gxg^{-1}x^{-1}=(0,g)(x,0)(0,g')(-x,0)\\
&=(\lambda_g(x),g)(-\lambda^{-1}_g(x),g')
=(\lambda_g(x)-x,0)=\lambda_g(x)-x=g*x.
\end{align*}

Let
\begin{equation}
\label{eq:repeated}
\begin{aligned}
    & X_0=X=A^1,\\
    & X_{n+1}=[G,X_n]=A^{n+2}\quad\text{for $n\geq0$.}
\end{aligned}
 \end{equation}

Thus the elements of the left series of $A$ are indeed iterated commutators of
the group $\Gamma$.  This observation has strong consequences. Our first
application is the following useful result, which was proved by Rump for
classical braces using different methods (see the corollary after Proposition 2
of~\cite{MR2278047}).

\begin{pro}
    \label{pro:pgroups}
    Let $p$ be a prime and $A$ be skew left brace of size $p^m$. Then $A$ is
    left nilpotent.
\end{pro}

\begin{proof}
    Let $G$ be the multiplicative group of $A$ and $X$ be the additive group of
    $A$. Since the semidirect product $\Gamma=A\rtimes G$ is a $p$-group, it is
    nilpotent. Thus there exists $k$ such that the $k$-repeated commutator
    $[\Gamma,\Gamma,\dots,\Gamma]$, where $\Gamma$ appears $k$-times, is trivial. Since
    \[
        A^k=[G,\dots,G,X]\subseteq [\Gamma,\dots,\Gamma],
    \]
    it follows that $A$ is left nilpotent.
\end{proof}

The following results follow immediately from theorems of P.  Hall:

\begin{lem}
    Let $A$ be finite skew left brace such that $A^3=0$. Then the additive group of
    $A^2$ is abelian. In fact $A^2$ is a trivial brace.
\end{lem}

\begin{proof}
The first part follows by \cite[Theorem~6]{Hall}. Note that
$(A^2)^2\subseteq A^3=0$, hence $a\circ b=a+b$ for all $a,b\in A^2$, and
the result follows.
\end{proof}

\begin{thm}
    \label{thm:A2}
    Let $A$ be left nilpotent skew left brace. Then the following statements
    hold:
    \begin{enumerate}
        \item The additive group of $A^2$ is locally nilpotent.
        \item The multiplicative group of $A/\ker\lambda$ is locally nilpotent.
    \end{enumerate}
\end{thm}

\begin{proof}
    Since each element of the left series of $A$ is a repeated commutator, the
    first claim follows from Hall's theorem \cite[Theorem~4]{Hall}.
    To prove the second claim, we use the notation above
    Proposition~\ref{pro:pgroups}. Let $K=[G,X]G\subseteq \Gamma$ and
    $H=[G,X]X$. Let $C$ be the centralizer of $H$ in $K$. Then by
    \cite[Theorem~4]{Hall}, $K/C$ is locally nilpotent. Note
    that, since $X$ is normal in $\Gamma$, $H=X$. Hence $G\cap C$ is the
    centralizer of $X$ in $G$, that is
    \begin{eqnarray*}G\cap C&=&\{ g\in G\mid gxg^{-1}=x, \text{ for all } x\in X\}\\
    &=&\{ g\in A\mid \lambda_g(x)=x, \text{ for all } x\in A\}=\ker\lambda.
    \end{eqnarray*}
     Thus $(GC)/C\cong G/(G\cap
     C)=G/\ker\lambda$ is locally nilpotent.
\end{proof}

We shall introduce some notation. Let $A$ be a skew left brace.
We denote by $\gamma^+(a,b)=a+b-a-b$ the commutator of $a,b$ in
$(A,+)$, for all $a,b\in A$. Let $B,C$ be two subgroups of $(A,+)$.
We define $\gamma^+(B,C)=\langle \gamma^+(b,c)\mid b\in B,\; c\in
C\rangle_+$, the additive subgroup generated by the elements
$\gamma^+(b,c)$, for $b\in B$ and $c\in C$. We also write
$*(a,b)=a*b$ for all $a,b\in A$, and $*(B,C)=B*C=\langle
*(b,c)\mid b\in B,\; c\in C\rangle_+$.
Let $M$ be the free monoid with basis $\{ \gamma^+, *\}$. Then the
elements of $M$ are words in the alphabet $\{ \gamma^+, *\}$, that
is, if $m\in M$ then $$m=\epsilon_1\epsilon_2\cdots \epsilon_s,$$
for some non-negative integer $s$ and $\epsilon_i\in \{ \gamma^+,
*\}$. In this case, we say that $m$ has degree $s$ and we write
$\deg(m)=s$. Furthermore, if $s>0$, we define
$$m(a_1a_2\dots a_{s+1})=\epsilon_1(a_1,\epsilon_2(a_2,\dots(\epsilon_s(a_s,a_{s+1}))\dots )),$$
for all $a_1,\dots ,a_{s+1}\in A$, and if $A_1,\dots, A_{s+1}$ are
subgroups of $(A,+)$, we define
    $$m(A_1A_2\dots A_{s+1})=\epsilon_1(A_1,\epsilon_2(A_2,\dots(\epsilon_s(A_s,A_{s+1}))\dots )).$$
Finally we denote by $A_1(t)$ the word $A_1A_1\dots A_1$ of length
$t$ in the letter $A_1$. We order $M$ with the degree-lexicographic
order, extending $*<\gamma^+$. Note that if $m_2>m_1$ are elements
of $M$, then
\begin{eqnarray*}\lefteqn{m_2(a_1\dots
a_{\deg(m_2)+1})+m_1(b_1\dots b_{\deg(m_1)+1})}\\
&=&m_1(b_1\dots b_{\deg(m_1)+1})+m_2(a_1\dots
a_{\deg(m_2)+1})\\
&&-\gamma^+(-m_1(b_1\dots b_{\deg(m_1)+1}),-m_2(a_1\dots
a_{\deg(m_2)+1})).
\end{eqnarray*}
In particular, the elements of the additive subgroup generated by
$$\{m(A(\deg(m)+1))\mid m\in M, \text{ with }\deg(m)\geq t\}$$
are of the form $a_1+a_2+\dots +a_s$, where $a_i\in
m_i(A(\deg(m_i)+1))$, $\deg(m_1)\geq t$ and $m_1<\dots <m_s$. We
denote this additive subgroup by
$$\sum_{\{m\in M\mid \deg(m)\geq t\}}m(A(\deg(m)+1)).$$

\begin{lem}
\label{lem:series} Let $A$ be a skew left brace. Let $G_1=\ker
\lambda$, and for $i>1$, let $G_i=[A,G_{i-1}]=\langle a\circ b\circ
a'\circ b'\mid a\in A,\; b\in G_{i-1}\rangle$. Let $M$ be the free
monoid with basis $\{ \gamma^+, *\}$. Then
$$G_n\subseteq \sum_{\{m\in M\mid \deg(m)\geq
n-1\}}m(A(\deg(m)+1)).$$
\end{lem}

\begin{proof}
Let $a\in \ker\lambda$ and $g\in A$. Note that
\begin{equation}
\label{eq:ker}
\begin{aligned}
g\circ a\circ g'\circ a'&=g\circ (a+ g')+ a'\\
&=g\circ a-g-a=g+\lambda_g(a)-g-a.
\end{aligned}
\end{equation}
We shall prove the result by induction on $n$. For $n=1$,
$$G_1=\ker\lambda\subseteq A=\sum_{\{m\in M\mid \deg(m)\geq 0\}}m(A(\deg(m)+1)).$$
Let $n>1$ and suppose that
$$G_{n-1}\subseteq \sum_{\{m\in M\mid \deg(m)\geq n-2\}}m(A(\deg(m)+1)).$$
Let $g\in A$ and $a\in G_{n-1}$. Then since $G_{n-1}$ is a subgroup
of $\ker\lambda$, by (\ref{eq:ker}) we have
$$g\circ a\circ g'\circ a'=g+\lambda_g(a)-g-a.
$$
Let $a_m\in m(A(\deg(m)+1))$ be such that
$$a=a_{m_1}+a_{m_2}+\dots +a_{m_{s}},$$
with $\deg(m_1)\geq n-2$ and $m_1<\dots <m_s$.  We have
\begin{eqnarray*}
g\circ a\circ g'\circ a'&=&g+\lambda_g(a)-g-a\\
&=&g+\lambda_g(a_{m_1}+a_{m_2}+\dots +a_{m_s})-g-(a_{m_1}+a_{m_2}+\dots +a_{m_s})\\
&=&g+\lambda_g(a_{m_1})+\dots +\lambda_g(a_{m_s})-g-(a_{m_1}+a_{m_2}+\dots +a_{m_s})\\
&=&g+(g*a_{m_1}+a_{m_1})+\dots +(g*a_{m_s}+a_{m_s})-g\\
&&-(a_{m_1}+a_{m_2}+\dots +a_{m_s}).
\end{eqnarray*}
We shall prove that
\begin{eqnarray*}
\lefteqn{g+(g*a_{m_1}+a_{m_1})+\dots +(g*a_{m_s}+a_{m_s})-g}\\
&&-(a_{m_1}+a_{m_2}+\dots +a_{m_s})\in\sum_{\{m\in M\mid \deg(m)\geq
n-1\}}m(A(\deg(m)+1))
\end{eqnarray*}
by induction on $s$. For $s=1$ we have
\begin{eqnarray*}
g+g*a_{m_1}+a_{m_1}-g-a_{m_1}&=&\gamma^+(g,g*a_{m_1})+g*a_{m_1}+g+a_{m_1}-g-a_{m_1}\\
&=&\gamma^+(g,g*a_{m_1})+g*a_{m_1}+\gamma^+(g,a_{m_1}).
\end{eqnarray*}
Since $\gamma^+(g,g*a_{m_1})\in \gamma^+*m_1(A(\deg(m_1)+3))$,
$g*a_{m_1}\in *m_1(A(\deg(m_1)+2))$ and $\gamma^+(g,a_{m_1})\in
\gamma^+m_1(A(\deg(m_1)+2))$, we have that
$$g+g*a_{m_1}+a_{m_1}-g-a_{m_1}\in\sum_{\{m\in M\mid
\deg(m)\geq n-1\}}m(A(\deg(m)+1)).$$

Suppose that $s>1$ and $g+(g*a_{m_1}+a_{m_1})+\dots
+(g*a_{m_{s-1}}+a_{m_{s-1}})-g-(a_{m_1}+a_{m_2}+\dots
+a_{m_{s-1}})\in \sum_{\{m\in M\mid \deg(m)\geq
n-1\}}m(A(\deg(m)+1)).$

We have that
\begin{eqnarray*}
\lefteqn{g+(g*a_{m_1}+a_{m_1})+\dots +(g*a_{m_s}+a_{m_s})-g}\\
&&-(a_{m_1}+a_{m_2}+\dots +a_{m_s})\\
&=&g+(g*a_{m_1}+a_{m_1})+\dots +(g*a_{m_{s-1}}+a_{m_{s-1}})\\
&&+g*a_{m_{s}}+a_{m_{s}}-g-a_{m_s}+g-g-(a_{m_1}+a_{m_2}+\dots +a_{m_{s-1}})\\
&=&g+(g*a_{m_1}+a_{m_1})+\dots +(g*a_{m_{s-1}}+a_{m_{s-1}})\\
&&+g*a_{m_{s}}-\gamma^+(-g,a_{m_{s}})-g-(a_{m_1}+a_{m_2}+\dots +a_{m_{s-1}})\\
&=&g+(g*a_{m_1}+a_{m_1})+\dots +(g*a_{m_{s-1}}+a_{m_{s-1}})-g\\
&&-(a_{m_1}+a_{m_2}+\dots +a_{m_{s-1}})\\
&&+g*a_{m_{s}}-\gamma^+(-g,a_{m_{s}})\\
&&-\gamma^+(a_{m_1}+a_{m_2}+\dots
+a_{m_{s-1}}+g,\gamma^+(-g,a_{m_{s}})-g*a_{m_{s}})\\
&&\in \sum_{\{m\in M\mid \deg(m)\geq n-1\}}m(A(\deg(m)+1)).
\end{eqnarray*}
Hence
$$g\circ a\circ g'\circ a'\in \sum_{\{m\in M\mid \deg(m)\geq
n-1\}}m(A(\deg(m)+1)).$$

Note that $\sum_{\{m\in M\mid \deg(m)\geq n-1\}}m(A(\deg(m)+1))$ is
a left ideal of $A$. Therefore
$$G_{n}\subseteq \sum_{\{m\in M\mid
\deg(m)\geq n-1\}}m(A(\deg(m)+1)),$$ and the result follows by
induction.
\end{proof}

The following result generalizes \cite[Theorem~1]{MR3814340}.

\begin{thm}
    \label{thm:left_nilpotent=nilpotent}
    Let $A$ be a finite skew left brace with nilpotent additive group. Then $A$ is
    left nilpotent if and only if the multiplicative group of $A$ is nilpotent.
\end{thm}

\begin{proof}
    Let us first assume that $(A,\circ)$ and $(A,+)$ are nilpotent.
    By Corollary~\ref{cor:product}, the skew left brace $A$ is the direct product of skew left braces with
    prime-power orders. By Proposition~\ref{pro:pgroups} all such skew left braces are
    left nilpotent, hence $A$ is left nilpotent by Lemma~\ref{lem:left_nilpotent:x}.

    Suppose now that $(A,+)$ is nilpotent and $A$ is left nilpotent.
    There exist positive integers $n_1,n_2$ such that $A^{n_1}=0$ and $\gamma_{n_2}^+(A)=0$, where $\gamma_{j+1}^+(A)=(\gamma^+)^j(A(j+1))$,
    using the notation above Lemma~\ref{lem:series}. By
    Theorem~\ref{thm:A2}, we know that the multiplicative group of
    $A/\ker\lambda$ is nilpotent. Let $\gamma_1(A)=A$ and for $i>1$
    let
    \[
    \gamma_i(A)=[A,\gamma_{i-1}(A)]=\langle a\circ b\circ
    a'\circ b'\mid a\in A,\; b\in \gamma_{i-1}(A)\rangle.
    \]
    Thus there exists a positive integer $k$ such that
    $\gamma_k(A)\subseteq\ker\lambda$. Using the notation in the
    proof of Lemma~\ref{lem:series}, we have that $\gamma_{k+j}(A)\subseteq
    G_{j+1}$ for every nonnegative integer $j$. Hence, by
    Lemma~\ref{lem:series}
    $$\gamma_{k+n_1n_2}(A)\subseteq G_{n_1n_2+1}\subseteq \sum_{\{m\in M\mid \deg(m)\geq
    n_1n_2\}}m(A(\deg(m)+1)).$$
    Let $m\in M$ be an element with $\deg(m)\geq n_1n_2$. Note that
    if $\gamma^+$ appears $t$ times in $m$, then $m(A(\deg(m)+1))\subseteq
    (\gamma^+)^t(A(t+1))$. In particular, if $t\geq n_2$, then
    $m(A(\deg(m)+1))=0$.
    Suppose that $\gamma^+$ appears at most $n_2-1$ times in $m$. In
    this case, there exist $m_1,m_2\in M$ such that
    $m=m_1(*)^{n_1}m_2$. In this case,
    \begin{align*}
    m(A(\deg(m)+1))&=m_1(*)^{n_1}(A(\deg(m_1)+n_1)m_2(A(\deg(m_2)+1)))\\
    &\subseteq  m_1(*)^{n_1}(A(\deg(m_1)+n_1+1))\\
    &=m_1(A(\deg(m_1)A^{n_1}))=0.
    \end{align*}
    Hence $\gamma_{k+n_1n_2}(A)=0$.
    Therefore the multiplicative group of $A$ is nilpotent, and the
    result follows.
\end{proof}

The assumption on the nilpotency of the additive group in
Theorem~\ref{thm:left_nilpotent=nilpotent} is needed (see
Example~\ref{ex:trivial}).

\begin{cor}
    Let $A$ be a finite skew left brace of size $p^n$ for some prime number $p$
    and some positive integer $n$.  Then either $A$ is the trivial brace of
    order $p$ or it is not simple.
\end{cor}

\begin{proof}
 By Theorem~\ref{thm:left_nilpotent=nilpotent}, $A$ is left nilpotent. In
 particular, if $A\neq 0$, then $A^2\neq A$. Since $A^2$ is an ideal either $A$
 is not simple or $A^2=0$. Assume that $A^2=0$. In this case, $a\circ b=a+b$
 for all $a,b\in A$. Therefore $[A,A]$ is a proper an ideal of $A$. Hence,
 either $A$ is not simple or $[A,A]=0$. Assume that $A^2=[A,A]=0$. In this case
 $A$ is a trivial brace and the result follows.
\end{proof}

\begin{lem}
    \label{lem:sylow_leftideals}
    Let $A$ be a finite skew left brace with nilpotent additive group. Let $p$
    and $q$ distinct prime numbers and let $P$ and $Q$ be Sylow subgroups of
    $(A,+)$ of sizes $p^n$ and $q^m$, respectively. Then $P$, $Q$ and $P+Q$ are
    left ideals of $A$.
\end{lem}

\begin{proof}
    Let us first prove that $P$ is a left ideal. Since $(A,+)$ is nilpotent,
    $P$ is a normal subgroup of $(A,+)$. Let $a\in A$ and $x\in P$. Then
    $\lambda_a(x)\in P$ since $\lambda_a$ is a group homomorphism. Similarly
    one proves that $Q$ is a left ideal. From this it follows that $P+Q$ is a
    left ideal.
\end{proof}

The following is based on~\cite[Theorem 5(1)]{MR3765444}.
However, the proof is completely different.

\begin{thm}
\label{thm:P*Q=0} Let $A$ be a finite skew left brace with nilpotent
additive group. Let $p$ and $q$ distinct prime numbers and let $A_p$
and $A_q$ be Sylow subgroups of $(A,+)$ of sizes $p^n$ and $q^m$,
respectively. If $p$ does not divide $q^t-1$ for all
$t\in\{1,\dots,m\}$, then $A_p*A_q=0$. In particular,
$\lambda_x(y)=y$ for all $x\in A_p$ and $y\in A_q$.
\end{thm}

\begin{proof}
By Lemma~\ref{lem:sylow_leftideals} $A_p$, $A_q$ and $A_p+A_q$
are left ideals of $A$. In particular,  $A_p+A_q$ is a skew subbrace
of $A$ and $A_p$ and $A_q$ are Sylow subgroups of
$(A_p+A_q,\circ)$. By Sylow's theorem, the number $n_p$ of Sylow
$p$-subgroups of the multiplicative group of $A_p+A_q$ is
\[
n_p=[A_p+A_q:N]\equiv 1\bmod p,
\]
where $N=\{g\in A_p+A_q:g\circ A_p\circ g'=A_p\}$ is the normalizer
of $A_p$ in the multiplicative group of $A_p+A_q$. Since
$[A_p+A_q:N]=q^s$ for some $s\in\{0,\dots,m\}$ and $p$ does not
divide $q^t-1$ for all $t\in\{1,\dots,n\}$, it follows that $s=0$
and hence $A_p$ is a normal subgroup of the multiplicative group of
$A_p+A_q$. Thus $A_p$ is an ideal of the skew left brace
$A_p+A_q$. Since $A_p$ is an ideal of $A_p+A_q$ and $A_q$ is a left
ideal, we have that $A_p*A_q\subseteq A_p\cap A_q=0$, and the result
follows.
\end{proof}

\begin{cor}
Let $A$ be a skew left brace of size $p_1^{\alpha_1}\cdots
p_k^{\alpha_k}$, where $p_1<p_2<\cdots<p_k$ are prime numbers and
$\alpha_1,\dots,\alpha_k$ are positive integers. Assume that the
additive group of $A$ is nilpotent. Let $A_j$ be the Sylow $p_j$-
subgroups of the additive group of $A$. Assume that, for some $j\leq k$,
$p_j$ does not divide $p_i^{t_i}-1$ for all $t_i\in\{
1,\dots \alpha_i\}$ for all $i\neq j$. Then $\Soc(A_j)\subseteq
\Soc(A)$.
\end{cor}

\begin{proof}
Write $A=A_1+\cdots+A_k$. Let $a\in \Soc(A_j)$ and $b\in A$. Hence
there exist elements $b_k\in A_k$ such that $b=b_1+\dots +b_k$. By
Theorem~\ref{thm:P*Q=0}, $\lambda_a(b_i)=b_i$, for all $i\neq j$.
Then $\lambda_a(b)=\lambda_a(b_1)+\dots +\lambda_a(b_k)=b_1+\dots
+b_k=b$ and hence $a\in Soc(A)$. Thus the result follows.
\end{proof}

\section{Braces with cyclic multiplicative group}
\label{cyclic}

As a consequence of a result of P. Hall one proves that the multiplicative
group of a finite left brace is solvable. The following example shows that
this fact does not hold for infinite left braces:

\begin{example}
    Let $K$ be a field of characteristic $0$. The Jacobson radical of
    $M_2(K[[x]])$ is $J=M_2(xK[[x]])$. Thus $(J,+,\circ)$ is a two-sided brace,
    where
    \[
        A\circ B=AB+A+B
    \]
    for all $A,B\in J$.
    The map $f\colon J\longrightarrow \GL(M_2(K[[x]]))$ defined by
    $f(A)=I_2+A$, for all $A\in J$, where $I_2$ is the identity $2\times 2$
    matrix, is a monomorphism of groups: $$f(A\circ
    B)=I_2+AB+A+B=(I_2+A)(I_2+B)=f(A)f(B).$$ By \cite[Lemma~2.8]{W}, the
    subgroup of $(J,\circ)$ generated by
    $$\left(\begin{array}{cc}
        0&0\\ x&0
    \end{array}\right), \quad
    \left(\begin{array}{cc} 0&x\\ 0&0
    \end{array}\right)$$
    is free of rank two. Therefore $(J,\circ)$ is not solvable.
\end{example}

The following result is due to Rump (see the proof of
\cite[Proposition~6]{MR2298848}).

\begin{thm}[Rump]\label{addZ}
    Let $A$ be a left brace of type
    $\Z$. Then either $A$ is the trivial brace isomorphic to $\Z$ or
    its multiplication is defined by
\begin{eqnarray}\label{mult}&&(ma)\circ (na)=((-1)^mn+m)a,
\end{eqnarray} for all $m,n\in \Z$, where $a\in A$ is a fixed generator of
    its additive group.
\end{thm}

%

Now we shall study the left braces with multiplicative group
isomorphic to $\Z$. Note that these are two-sided braces. Thus this
is equivalent to the study of the Jacobson radical rings such that
its circle group is isomorphic to $\Z$.

\begin{lem}\label{fg}
Let $R$ be a Jacobson radical ring with its circle group isomorphic
to $\Z$. Then the additive group of $R$ is finitely generated.
\end{lem}

\begin{proof}
Let $x\in R$ be a generator of the circle group $(R,\circ)$ of $R$.
Let $y\in R$ the inverse of $x$ in $(R,\circ)$. Then
$$R=\left\{\sum_{i=1}^n {n\choose i} x^{i}\mid n\geq 1\right\}\cup \left\{\sum_{i=1}^n {n\choose i} y^{i}\mid n\geq 1\right\}\cup \{ 0\}.$$
Since $x+y\in R$ one of the following conditions holds:
\begin{itemize}
\item[(1)] $x+y=0$. In this case, $0=xy=-x^2$, thus
$R=\{ zx\mid z\in \Z\}$ and the additive group of $R$ is isomorphic
to $\Z$.
\item[(2)] There  exists a positive integer $n$ such that
$$x+y=\sum_{i=1}^n {n\choose i} x^{i}.$$
Since $y\neq 0$, we have that $n>1$. Then
$$0=xy+x+y=x\left(-x+\sum_{i=1}^n {n\choose i} x^{i}\right)+\sum_{i=1}^n {n\choose i} x^{i}.$$
Hence $R=x\Z+x^2\Z+\dots +x^n\Z$.
\item[(3)]  There  exists a positive integer $n$ such that
$$x+y=\sum_{i=1}^n {n\choose i} y^{i}.$$
Since $x\neq 0$, we have that $n>1$. Then
$$0=yx+x+y=y\left(-y+\sum_{i=1}^n {n\choose i} y^{i}\right)+\sum_{i=1}^n {n\choose i} y^{i}.$$
Hence $R=y\Z+y^2\Z+\dots +y^n\Z$.
\end{itemize}
Therefore the result follows.
\end{proof}

\begin{thm}\label{multZ}
Let $R$ be a Jacobson radical ring with its circle group isomorphic
to $\Z$. Then $ab=0$ for all $a,b\in R$.
\end{thm}
\begin{proof}
Let $p$ be a prime number. Since the additive group of $R$ is
infinite and finitely generated,  $pR$ is a proper ideal of $R$ and
$R/pR$ has order $p^m$ for some positive integer $m$. Since the only
simple left brace of order a power of $p$ is the trivial brace of
order $p$, there exists a maximal ideal $I_p$ of $R$ such that
$pR\subseteq I_p$ and $R/I_p$ has order $p$. Let $x$ be a generator
of the circle group of $R$. Then it is clear that the circle group
of $I_p$ is generated by $x^{\circ p}$, (where $x^{\circ
n}=x\circ\dots\circ x$ ($n$ times)). Let $a,b\in R$. We have that
$(a+I_p)(b+I_p)=I_p$ because $R/I_p$ is a ring with zero
multiplication. Hence $ab\in I_p$, for all prime numbers $p$. Now
$$\bigcap_{p \text{ prime}}I_p=\bigcap_{p \text{ prime}}\{x^{\circ zp}\mid z\in \Z\}=\{0\}.$$
Therefore $ab=0$, and the result follows.
\end{proof}

As a consequence we have the following result.

\begin{thm}
    \label{thm:Z}
    Let $A$ be a left brace with multiplicative group isomorphic to $\Z$. Then
    $A$ is a trivial brace, in particular the additive group of $A$ is
    isomorphic to $\Z$.
\end{thm}

A natural question arises: Is it possible to extend Theorem~\ref{thm:Z} to skew
left braces? One can prove Theorem~\ref{thm:Z} for skew left braces of finite
multipermutation level. However, the following result shows that nothing new is
covered in this case.

\begin{thm}
  \label{thm:ZrightNilpotent}
  Let $A$  be a skew left brace with multiplicative group isomorphic to $\Z$. Then
  $A$ has finite multipermutation level if and only if $A$ is of abelian type.
\end{thm}
\begin{proof}
    We proved in Theorem~\ref{thm:Z} that skew left braces of abelian type with
    multiplicative group isomorphic to $\Z$ have finite multipermutation level.
    Let us assume that $A$ has finite multipermutation level. Let
    $c\in\Soc(A)\setminus\{0\}$. Then $c^{\circ k}=kc$ for all $k\in\N$. Since
    $(A,\circ)$ is torsion-free, $c^{\circ k}\ne c^{\circ l}$ if $k\ne l$.
    Observe that $\lambda_a(c^{\circ k})=c^{\circ k}$ for all $a\in
    A$ and $k\in\N$, because $(A, \circ )$ is commutative.
    For
    $k>0$ let $I_k$ be the ideal of $A$ generated by $c^{\circ k}$.  Then
    $I_k=\{klc:l\in\Z\}$ and $\bigcap_{k>0} I_k=\{0\}$. Let $k>0$. Since $c\ne0$, $A/I_k$
    is a finite skew left brace of finite multipermutation level. By
    Theorem~\ref{thm:mpl&right_nilpotent}, $A/I_k$ is of nilpotent type. Thus
    Corollary~\ref{cor:product} implies that $A/I_k$ is a direct product of skew left braces
    of prime-power size. Using results of T. Kohl~\cite{MR1644203} quoted in
    \cite[Example A.7]{MR3763907}, such skew left braces are either of abelian type or of
    size $2^\alpha$ for some $\alpha$. Let us assume that $A$ is not of abelian
    type and let $a,b\in A$ be such that $a+b-a-b\ne0$. For each $k>0$ there exists $m(k)\in\N$ such that
    \[
        2^{m(k)}(a+b-a-a)\in I_k.
    \]
    Since $I_k=\{klc:l\in\Z\}$, for each $a,b\in A$ and each $k>0$ there are $m(k)\in\N$ and $q(k)\in\Z$ such that
    \[
        2^{m(k)}(a+b-a-b)=q(k)kc.
    \]
    Let $k$ be an odd prime number coprime with $3q(3)$. Then
    \[
        (2^{m(3)}q(k)k-2^{m(k)}q(3)3)(a+b-a-b)=0.
    \]
    Since $k$ is an odd prime number coprime with $3q(3)$, it follows that
    there exists $n\ne0$ such that $n(a+b-a-b)=0$. Then $nq(3)3c=0$, a
    contradiction. Therefore $A$ is a skew left brace of abelian type.
\end{proof}

\begin{rem}
In~\cite{Greenfeld},
Greenfeld showed that adjoint groups of Jacobson radical and not nilpotent
algebras cannot be finite products of cyclic groups. His results are general
but hold only for algebras over fields; therefore our results do not follow
from~\cite{Greenfeld}.
\end{rem}

The following result shows that Theorem~\ref{thm:Z} cannot be
extended to skew left braces. Recall that the \emph{infinite
dihedral group} is the group
\[
\D_{\infty}=\langle r,s:srs=r^{-1},\,s^2=1\rangle\simeq\Z\rtimes\Z/(2).
\]
\begin{thm}
    \label{thm:infinite_dihedral}
    There exists a skew left brace with multiplicative group isomorphic to $\Z$ and
    additive group isomorphic to the infinite dihedral group $\D_\infty$.
\end{thm}
\begin{proof}
Let $G=\langle g\rangle\simeq\Z$. A direct calculation shows that
the operations
\[
    g^k+g^l=g^{k+(-1)^kl},\quad
    g^k\circ g^l=g^{k+l},\quad k,l\in\Z,
\]
turn $G$ into a skew left brace with additive group isomorphic to the infinite
dihedral group $\D_{\infty}$ and multiplicative group isomorphic to $\Z$.
\end{proof}

\section{Indecomposable solutions}
\label{indecomposable}

Let $A$ be a skew left brace and $a\in A$. We say that the skew left
brace $A$ is generated by $a$ if $A$ is the smallest sub skew left
brace of $A$ containing $a$. Let $(X, r)$ be a non-degenerate
set-theoretic solution of the Yang--Baxter equation, where $r(x, y)
= (\sigma_x(y), \tau_y(x))$. Recall from \cite{MR1848966} that $(X,
r)$ is said to be decomposable if there exist disjoint non-empty
subsets $X_1$ and $X_2$ of $X$ such that $X = X_1\cup X_2$ and
$r(X_i\times X_j) = X_j\times X_i$ for all $i, j\leq 2$. If it is
not possible to find such subsets $X_1$ and $X_2$ of $X$, the
solution $(X, r)$ is said to be indecomposable.  By the orbit of an
element $z\in X$ we will mean the smallest subset $Y$ of $X$ such
that $z\in Y$ and
$\sigma_x(y),\sigma^{-1}_x(y),\tau_x(y),\tau^{-1}_x(y)\in Y$, for
all $y\in Y$ and $x\in X$. That is, if $H$ is the subgroup of the
symmetric group $\Sym(X)$ over $X$ generated by $\sigma_x,\tau_x$,
for all $x\in X$, then the orbit of $z\in X$ is $Y=\{ h(z)\colon
h\in H\}$. Note that a non-degenerate set-theoretic solution $(X,
r)$ of the Yang--Baxter equation can be decomposed into orbits
$X_i$, for $i\in I$, such that $r(X_i\times X_j) = X_j\times X_i$
for all $i, j\in I$. Each restriction $(X_i, r|_{X_i\times X_i})$ is
again a non-degenerate set-theoretic solution of the Yang--Baxter
equation. However, such restricted solutions need not to be
indecomposable.

\begin{example} Let $(A,+,\cdot)$ be a commutative nilpotent ring
with generators $x, y$ and relations $x + x = y + y = 0$, $x^2 = y^2
= 0$. Let $(A,+,\circ)$ be the associated brace and $(A, r_A)$ be
the associated involutive non-degenerate set-theoretic solution.
Then $Y = \{ x, x + yx\}$ is an orbit. Observe that the solution
$(Y, r|_{Y\times Y})$ is decomposable and $Y = \{x\}\cup\{ x + xy\}$
is the decomposition of $Y$ into its orbits.
\end{example}

Recall that if $A$ is a skew left brace, then its associated
solution $(A,r_A)$ is defined by $r_A(a,b)=(\sigma_a(b),\tau_b(a))$,
for all $a,b\in A$, where $\sigma_a(b)=\lambda_a(b)$ and
$\tau_b(a)=\lambda_a(b)'\circ a\circ b$ (see \cite{MR3647970}). In
\cite[Theorem~3.1]{MR3647970} it is proved that $(A,r_A)$ is a
non-degenerate set-theoretic solution of the Yang--Baxter equation.

\begin{rem}\label{orbits}
Let $(X,r)$ be an involutive non-degenerate set-theoretic solution
of the Yang--Baxter equation. We write
$r(x,y)=(\sigma_x(y),\tau_y(x))$. Since $r$ is involutive we have
that $\tau_y(x)=\sigma^{-1}_{\sigma_x(y)}(x)$ and
$\sigma_a(b)=\tau^{-1}_{\tau_b(a)}(b)$. Note that the orbit of $x\in
X$ is
\[
    O_x=\{
\sigma_{y_1}\tau_{y_2}\dots\sigma_{y_{2m-1}}\tau_{y_{2m}}(x)\colon
y_1,\dots y_{2m}\in X\cup \{ 0\}\},
\]
where $0\notin X$ and
$\sigma_0=\tau_0=\id_X$. This is because
$$\sigma^{-1}_y(z)=\sigma^{-1}_{\sigma_z(\sigma^{-1}_z(y))}(z)=\tau_{\sigma^{-1}_z(y)}(z)\in O_x,$$
and
$$\tau^{-1}_y(z)=\tau^{-1}_{\tau_z(\tau^{-1}_z(y))}(z)=\sigma_{\tau^{-1}_z(y)}(z)\in O_x,$$
for all $z\in O_x$ and all $y\in X$.

Therefore our definition of orbit of an element of a non-degenerate
set-theoretic solution of the Yang--Baxter equation coincides with
the definition of orbit in \cite[Section~2.1]{MR3771874} in the
involutive case. It is easy to see that these definitions of orbits
also coincide for finite non-degenerate set-theoretic solutions of
the Yang--Baxter equation. However our definition of orbit does not
coincide with the definition of orbit in
\cite[Section~2.1]{MR3771874} for arbitrary infinite non-degenerate
set-theoretic solution of the Yang--Baxter equation, as the
following example shows.
\end{rem}

\begin{exa}
Consider the map $r\colon\Z\times \Z\rightarrow \Z\times \Z$ defined
by $r(a,b)=(b+1,a+1)$, for all $a,b\in\Z$. It is easy to check that
$(\Z,r)$ is a non-degenerate set-theoretic solution of the
Yang--Baxter equation. With our definition of orbit, the orbit of
every element $a\in \Z$ is $\Z$. With the definition of orbit in
\cite[Section~2.1]{MR3771874}, the orbit of $a\in \Z$ is
\[
X_a=\{a+n\colon n \text{ is a non-negative integer}\}.
\]
Note that
the restriction $r'=r|_{X_a\times X_a}\colon X_a\times
X_a\rightarrow X_a\times X_a$ of $r$ to $X_a\times X_a$ is
non-bijective map. In fact $(X_a, r')$ is a non-bijective
set-theoretic solution of the Yang--Baxter equation. Furthermore, if
we write $r'(b,c)=(\sigma_b(c),\tau_c(b))$, for all $b,c\in X_a$,
then $\sigma_b\colon X_a\rightarrow X_a$ is not bijective.
\end{exa}

\begin{pro}\label{1new}
Let $A$ be a skew left brace generated (as a skew left brace) by an
element $x$. Let $X=\{\lambda_a(x)\colon a\in A\}$. Then $A=\langle
X\rangle_+=\langle X\rangle_{\circ}$.
\end{pro}

\begin{proof}
Note that $x=\lambda_0(x)\in X$. Since
$\lambda:(A,\circ)\rightarrow\Aut(A,+)$ is a homomorphism of groups,
it is clear that $\lambda_a(\langle X\rangle_+)\subseteq \langle
X\rangle_+$ for all $a\in A$. Let $y,z\in \langle X\rangle_+$. We
have that
$$y\circ z'=y+\lambda_y(z')=y+\lambda_y(-\lambda_{z'}(z))\in \langle X\rangle_+.$$
Hence $\langle X\rangle_+$ is a sub skew left brace of $A$
containing $x$ and thus $A=\langle X\rangle_+$.

Let $t\in \langle X\rangle_{\circ}$. Then
$t=\lambda_{a_1}(x)^{\varepsilon_1}\circ\dots\circ
\lambda_{a_n}(x)^{\varepsilon_n}$, for some $a_j\in A$ and
$\varepsilon_j\in \{ {}',1\}$ (where $a^1=a$). We shall prove that
$\lambda_a(t)\in \langle X\rangle_{\circ}$ by induction on $n$. For
$n=1$, we may assume that $t=\lambda_{a_1}(x)'$. In this case
\begin{align*}\lambda_a(t)&=\lambda_a(-\lambda_t(\lambda_{a_1}(x)))=-\lambda_{a\circ
t\circ a_1}(x)\\
&=\lambda_{(-\lambda_{a\circ t\circ a_1}(x))'}(\lambda_{a\circ
t\circ a_1}(x))'\in \langle X\rangle_{\circ}
\end{align*}
Suppose that $n>1$ and that
$\lambda_b(\lambda_{b_1}(x)^{\nu_1}\circ\dots\circ
\lambda_{b_{n-1}}(x)^{\nu_{n-1}})\in \langle X\rangle_{\circ}$, for
all $b,b_1,\dots, b_{n-1}\in A$ and $\nu_j\in\{ {}',1\}$. Thus
\begin{align*}
\lambda_a(t)&=\lambda_a(\lambda_{a_1}(x)^{\varepsilon_1}+\lambda_{\lambda_{a_1}(x)^{\varepsilon_1}}(\lambda_{a_2}(x)^{\varepsilon_2}\circ\dots\circ
\lambda_{a_n}(x)^{\varepsilon_n}))\\
&=\lambda_a(\lambda_{a_1}(x)^{\varepsilon_1})+\lambda_{a\circ\lambda_{a_1}(x)^{\varepsilon_1}}(\lambda_{a_2}(x)^{\varepsilon_2}\circ\dots\circ
\lambda_{a_n}(x)^{\varepsilon_n})\\
&=\lambda_a(\lambda_{a_1}(x)^{\varepsilon_1})\circ\lambda_{(\lambda_a(\lambda_{a_1}(x)^{\varepsilon_1}))'\circ
a\circ\lambda_{a_1}(x)^{\varepsilon_1}}(\lambda_{a_2}(x)^{\varepsilon_2}\circ\dots\circ
\lambda_{a_n}(x)^{\varepsilon_n})\in \langle X\rangle_{\circ},
\end{align*}
by the inductive hypothesis. Therefore $\lambda_a(\langle
X\rangle_{\circ})\subseteq \langle X\rangle_{\circ}$, for all $a\in
A$.

Let $t,z\in \langle X\rangle_{\circ}$. We have that
$-t+z=\lambda_t(t'\circ z)\in \langle X\rangle_{\circ}$. Hence
$\langle X\rangle_{\circ}$ is a sub skew left brace of $A$
containing $x$ and thus $A=\langle X\rangle_{\circ}$ and the result
follows.
\end{proof}

As a consequence we obtain a generalization of \cite[Theorem~5.4]{MR3771874} to
skew left braces. In particular, by Remark~\ref{orbits},  we answer in positive
\cite[Question~5.6]{MR3771874}.

\begin{pro}
    Let $B$ be a skew left brace and let $x\in B$. Let $A = B(x)$ be the
    smallest sub skew left brace of $B$ containing $x$. Let $(A, r_A)$ be the
    solution associated to the skew left brace $A$ and let $X$ be the orbit of
    $x$ in $(A, r_A)$. Then $(X, r_A|_{X\times X})$ is indecomposable.
\end{pro}

\begin{proof}
Recall that $r_A(a,b)=(\sigma_a(b), \tau_b(a))$, where
$\sigma_a(b)=\lambda_a(b)$ and $\tau_b(a)=\lambda_a(b)'\circ a\circ
b$, for all $a,b\in A$. By \cite[Corollary~1.10]{MR3647970}, the map
$\lambda\colon (A,\circ)\rightarrow \Aut(A,+)$ defined by
$\lambda(a)=\lambda_a$, for all $a\in A$, is a homomorphism of
groups. By \cite[Lemma~2.4]{Bachiller3}, the map $\tau\colon
(A,\circ)\rightarrow \Sym(A)$ defined by $\tau(a)=\tau_a$ is an
antihomomorphism of groups.

Let $X_1=\{ \lambda_a(x)\colon a\in A\}$. Note that $X_1\subseteq
X$. By Proposition~\ref{1new}, $A=\langle X_1\rangle_\circ$. Let
$z,t\in X$. Hence there exist $a_1,\dots, a_{2n},b_1,\dots
,b_{2m}\in X_1\cup X'_1\cup \{ 1\}$, such that
$$z=\lambda_{a_1}\tau_{a_2}\cdots \lambda_{a_{2n-1}}\tau_{a_{2n}}(x)\;\mbox{ and }\;
t=\lambda_{b_1}\tau_{b_2}\cdots
\lambda_{b_{2m-1}}\tau_{b_{2m}}(x),$$ where $X'_1=\{ y'\colon y\in
X_1\}$. Thus
$$
t=\lambda_{b_1}\tau_{b_2}\cdots
\lambda_{b_{2m-1}}\tau_{b_{2m}}\tau_{a'_{2n}}\lambda_{a'_{2n-1}}\cdots
\tau_{a'_{2}}\lambda_{a'_{1}}(z).$$ Therefore the result follows.
\end{proof}

\begin{pro}
 Let $(X, r)$ be a non-degenerate set-theoretic solution of the Yang--Baxter equation,
 and suppose that $(X, r)$ is decomposable with $X=Y\cup Z$, then $Y$ and $Z$
 are unions of orbits. In particular, every non-degenerate solution with a unique
 orbit is indecomposable.
\end{pro}

\begin{proof}
 If $(X,r)$ is a decomposable non-degenerate solution, then $X$ is the union of
 two disjoint non-empty sets $Y$ and $Z$ such that $r(Y\times Y)=Y\times Y$,
 $r(Z\times Z)=Z\times Z$, $r(Y\times Z)=Z\times Y$ and $r(Z\times Y)=Y\times
 Z$.  Therefore $r(X\times Y)=Y\times X$ and $r(Y\times X)=X\times Y$. We write
 $r(x,y)=(\sigma_x(y),\tau_y(x))$. Let $y\in Y$ and $x\in X$. Hence have
 $\sigma_x(y)\in Y$ and $\tau_x(y)\in Y$. Similarly $\sigma_x(z),\tau_x(z)\in
 Z$ for all $z\in Z$ and $x\in X$. Since $\sigma_x(\sigma^{-1}_x(y))=y\in Y$
 for all $y\in Y$ and $x\in X$, we have that $\sigma^{-1}_x(y)\in X\setminus
 Z=Y$. Similarly, $\tau^{-1}_x(y)\in Y$ for all $y\in Y$ and all $x\in X$.
 Hence $Y$ and $Z$ are unions of orbits of elements.
 In particular, every non-degenerate solution with a unique orbit is
 indecomposable.
\end{proof}

\section*{Acknowledgements}

The first named-author was partially supported by the grants MINECO-FEDER
MTM2017-83487-P and AGAUR 2017SGR1725(Spain).  The second-named author is
supported by the ERC Advanced grant 320974 and EPSRC Programme Grant
EP/R034826/1.  The third-named author is supported by PICT-201-0147, MATH-AmSud
17MATH-01 and ERC Advanced grant 320974. We thank E. Acri, N. Byott and E.
Jespers for useful comments.

\bibliographystyle{abbrv}
\bibliography{refs}


\end{document}